\documentclass[a4paper,leqno]{amsart}
\usepackage{amsmath,amssymb,amscd,mathrsfs}
\usepackage{graphicx,subfigure,multirow}
\usepackage{mathptmx,courier}
\usepackage[scaled=0.95]{helvet}
\usepackage{nicefrac}

\graphicspath{{./figures/}}

\usepackage{ifpdf}
\ifpdf
\usepackage[pdftex,
            a4paper,
            pdfauthor={Michael Eisermann, Christoph Lamm},
            pdftitle={A refined Jones polynomial for symmetric unions}
            pdfcreator={LaTeX with hyperref},
            pdfproducer={LaTeX with hyperref},
            pdfstartview=FitH,
            bookmarksopen=true,
            bookmarksnumbered,
            colorlinks=true, 
            anchorcolor=black, 
            linkcolor=black, 
            urlcolor=black,
            citecolor=black, 
            pagecolor=black, 
            filecolor=black, 
            menucolor=black, 
            ]{hyperref}
\else
\usepackage[dvips,
            a4paper,
            pdfauthor={Michael Eisermann, Christoph Lamm},
            pdftitle={A refined Jones polynomial for symmetric unions}
            pdfcreator={LaTeX with hyperref},
            pdfproducer={LaTeX with hyperref},
            pdfstartview=FitH,
            bookmarksopen=true,
            bookmarksnumbered,
            colorlinks=true, 
            anchorcolor=black, 
            linkcolor=black, 
            urlcolor=black,
            citecolor=black, 
            pagecolor=black, 
            filecolor=black, 
            menucolor=black, 
            ]{hyperref}
\fi

\title{A refined Jones polynomial for symmetric unions}

\author{Michael Eisermann}
\address{Institut f\"ur Geometrie und Topologie, Universit\"at Stuttgart, Germany}
\email{Michael.Eisermann@mathematik.uni-stuttgart.de}
\urladdr{www.igt.uni-stuttgart.de/eiserm}

\author{Christoph Lamm}
\address{R{\"u}ckertstra{\ss}e 3, 65187 Wiesbaden, Germany}
\email{Christoph.Lamm@web.de}

\date{first version September 2006; this version compiled \today}

\theoremstyle{plain}
\newtheorem{theorem}{Theorem}[section]
\newtheorem{lemma}[theorem]{Lemma}
\newtheorem{proposition}[theorem]{Proposition}
\newtheorem{corollary}[theorem]{Corollary}

\theoremstyle{remark}
\newtheorem{question}[theorem]{Question}
\newtheorem{remark}[theorem]{Remark}
\newtheorem{example}[theorem]{Example}
\newtheorem*{notation}{Notation}

\newtheorem{definition}[theorem]{Definition}

\newcommand{\sref}[1]{\textsection\ref{#1}}
\newcommand{\fref}[1]{Fig.\,\ref{#1}}

\newcommand{\figbox}[3]{\parbox[c][#1][c]{#2}{\centering#3}}

\newcommand{\Z}{\mathbb{Z}}
\newcommand{\R}{\mathbb{R}}
\newcommand{\C}{\mathbb{C}}
\newcommand{\CP}{\mathbb{CP}}

\newcommand{\D}{\mathbb{D}}
\renewcommand{\S}{\mathbb{S}}

\newcommand{\udiagrams}{\mathscr{D}}
\newcommand{\odiagrams}{\smash{\vec{\mathscr{D}}}}

\newcommand{\bracket}[1]{\left\langle #1 \right\rangle}
\newcommand{\lk}{\operatorname{lk}}
\newcommand{\arf}{\operatorname{arf}}
\newcommand{\onehalf}{{\nicefrac{1}{2}}}
\newcommand{\threehalves}{{\nicefrac{3}{2}}}
\newcommand{\fivehalves}{{\nicefrac{5}{2}}}

\newcommand{\isoto}{\mathrel{\xrightarrow{_\sim}}}

\newcommand{\cpic}[1]{\smash{\raisebox{-0.65ex}{\includegraphics[height=2.5ex]{#1}}}}

\newsavebox{\boxocr}\savebox{\boxocr}{\cpic{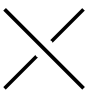}}
\newsavebox{\boxucr}\savebox{\boxucr}{\cpic{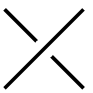}}
\newsavebox{\boxhcr}\savebox{\boxhcr}{\cpic{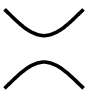}}
\newsavebox{\boxvcr}\savebox{\boxvcr}{\cpic{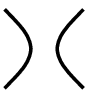}}

\newsavebox{\boxaxis}\savebox{\boxaxis}{\cpic{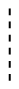}}
\newsavebox{\boxaocr}\savebox{\boxaocr}{\cpic{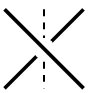}}
\newsavebox{\boxaucr}\savebox{\boxaucr}{\cpic{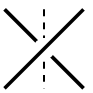}}
\newsavebox{\boxahcr}\savebox{\boxahcr}{\cpic{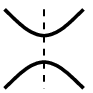}}
\newsavebox{\boxavcr}\savebox{\boxavcr}{\cpic{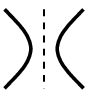}}
\newsavebox{\boxascr}\savebox{\boxascr}{\cpic{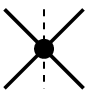}}

\newcommand{\aocr}{\usebox{\boxaocr} }
\newcommand{\aucr}{\usebox{\boxaucr} }
\newcommand{\ahcr}{\usebox{\boxahcr} }
\newcommand{\avcr}{\usebox{\boxavcr} }
\newcommand{\ascr}{\usebox{\boxascr} }

\newsavebox{\boxpcr}\savebox{\boxpcr}{\cpic{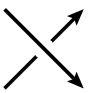}}
\newsavebox{\boxncr}\savebox{\boxncr}{\cpic{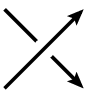}}
\newsavebox{\boxrcr}\savebox{\boxrcr}{\cpic{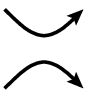}}

\newcommand{\pcr}{\usebox{\boxpcr} }
\newcommand{\ncr}{\usebox{\boxncr} }

\newsavebox{\boxapcr}\savebox{\boxapcr}{\cpic{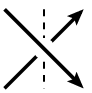}}
\newsavebox{\boxancr}\savebox{\boxancr}{\cpic{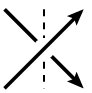}}
\newsavebox{\boxarcr}\savebox{\boxarcr}{\cpic{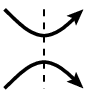}}

\newcommand{\apcr}{\usebox{\boxapcr} }
\newcommand{\ancr}{\usebox{\boxancr} }

\newcommand{\pic}[1]{\raisebox{-0.9ex}{\includegraphics[height=3ex]{#1}}}

\begin{document} 

\begin{abstract}
  Motivated by the study of ribbon knots we explore symmetric unions, 
  a beautiful construction introduced by Kinoshita and Terasaka in 1957.
  For symmetric diagrams we develop a two-variable refinement $W_D(s,t)$ 
  of the Jones polynomial that is invariant under symmetric Reidemeister moves.
  Here the two variables $s$ and $t$ are associated to the two types 
  of crossings, respectively on and off the symmetry axis.
  From sample calculations we deduce that a ribbon knot can have essentially
  distinct symmetric union presentations even if the partial knots are the same.

  If $D$ is a symmetric union diagram representing a ribbon knot $K$,
  then the polynomial $W_D(s,t)$ nicely reflects the geometric properties of $K$.
  In particular it elucidates the connection between 
  the Jones polynomials of $K$ and its partial knots $K_\pm$:
  we obtain $W_D(t,t) = V_K(t)$ and $W_D(-1,t) = V_{K_-}(t) \cdot V_{K_+}(t)$,
  which has the form of a symmetric product $f(t) \cdot f(t^{-1})$ 
  reminiscent of the Alexander polynomial of ribbon knots.
\end{abstract}





\maketitle

\vspace{-5mm}


\section{Introduction and outline of results} \label{sec:Introduction}

A knot diagram $D$ is said to be a \emph{symmetric union} if it is 
obtained from a connected sum of a knot $K_+$ and its mirror image $K_-$ 
by inserting an arbitrary number of crossings on the symmetry axis.
Figure \ref{fig:Knot-9_27} displays two examples with $K_\pm = 5_2$.
(We shall give detailed definitions in \sref{sec:Definitions}.)
Reversing this construction, the knots $K_\pm$ can be recovered 
by cutting along the axis; they are called the \emph{partial knots} of $D$.

\begin{figure}[hbtp]
  \centering
  \includegraphics[height=24ex]{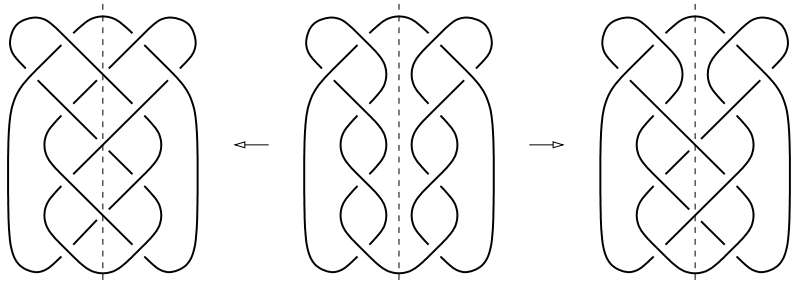}
  \caption{Two symmetric union presentations of the ribbon knot $9_{27}$ (left and right) 
    obtained from the connected sum of the partial knots $K_\pm = 5_2$ (middle)
    by inserting crossings on the symmetry axis}
  \label{fig:Knot-9_27}
\end{figure} 

The two outer diagrams of \fref{fig:Knot-9_27} both represent 
the knot $9_{27}$, which means that they are equivalent 
via the usual Reidemeister moves, see \cite[Fig.\ 8]{EisermannLamm:2007}.  
Are they equivalent through symmetric diagrams?
In the sequel we construct a two-variable refinement $W_D(s,t)$ of 
the Jones polynomial, tailor-made for symmetric union diagrams $D$
and invariant under symmetric Reidemeister moves.
This allows us to show that there cannot exist any symmetric 
transformation between the above diagrams, in other words, 
every transformation must break the symmetry in some intermediate stages.

\subsection{Motivation and background}

Symmetric unions were introduced by Kinoshita 
and Terasaka \cite{KinoshitaTerasaka:1957} in 1957.  
Apart from their striking aesthetic appeal, they 
appear naturally in the study of ribbon knots,
initiated at the same time by Fox and Milnor
\cite{FoxMilnor:1957,Fox:1962,FoxMilnor:1966}.
While ribbon and slice knots have received much attention 
over the last 50 years \cite{Livingston:2005}, 
the literature on symmetric unions remains scarce.
We believe, however, that the subject is worthwhile in its own right,
and also leads to productive questions about ribbon knots.

It is an old wisdom that, \emph{algebraically}, 
a ribbon knot $K$ resembles a connected sum $K_+ \sharp K_-$
of some knot $K_+$ with its mirror image $K_-$.
This is \emph{geometrically} modelled by symmetric unions:
it is easy to see that every symmetric union 
represents a ribbon knot (\sref{sub:SymmetricUnions}).
The converse question is still open; some affirmative 
partial answers are known \cite{EisermannLamm:2007}.
For example, all ribbon knots up to $10$ crossings and 
all two-bridge ribbon knots can be represented as symmetric unions.

Besides the problem of \emph{existence} it is natural to consider 
the question of \emph{uniqueness} of symmetric union representations.
Motivated by the task of tabulating symmetric union diagrams 
for ribbon knots, we were led to ask when two such diagrams 
should be regarded as equivalent.  
A suitable notion of symmetric Reidemeister moves has been 
developed in \cite[\textsection2]{EisermannLamm:2007}.
Empirical evidence suggested that ribbon knots can have
essentially distinct symmetric union representations,
even if the partial knots $K_\pm$ are the same. 
With the tools developed in the present article 
we can solve this problem in the affirmative 
for the knot $9_{27}$ as in \fref{fig:Knot-9_27},
and indeed for an infinite family of two-bridge 
ribbon knots (\sref{sub:TwoBridgeRibbonKnots}).

\subsection{A refined Kauffman bracket}

As our main tool we develop a two-variable refinement of the Jones polynomial 
that nicely reflects the geometric properties of symmetric unions.
Since skein relations are local and do not respect global symmetry conditions,
we are led to consider arbitrary diagrams for the following construction.

\begin{definition}[refined bracket polynomial] \label{def:TwoVariableBracket}
  Consider the plane $\R^2$ with vertical axis $\{0\} \times \R$ and let $\udiagrams$ 
  be the set of planar link diagrams that are transverse to the axis.
  The Kauffman bracket \cite{Kauffman:1987} can be refined to a two-variable 
  invariant $\udiagrams \to \Z(A,B)$, $D \mapsto \bracket{D}$,
  according to the following skein relations:
  \begin{itemize}

  \item
    For every crossing off the axis we have the usual skein relation 
    \renewcommand{\pic}[1]{\bracket{\raisebox{-0.9ex}{\includegraphics[height=3ex]{#1}}}}
    \[
    \tag{A} \label{eq:SkeinA}
    \pic{cross-o} = A^{+1} \pic{cross-h} + A^{-1} \pic{cross-v} .
    \]

  \item
    For every crossing on the axis we have an independent skein relation 
    \renewcommand{\pic}[1]{\bracket{\raisebox{-0.9ex}{\includegraphics[height=3ex]{#1}}}}
    \[
    \tag{B} \label{eq:SkeinB}
    \begin{aligned}
      \pic{cross-o-axis} & = B^{+1} \pic{cross-h-axis} + B^{-1} \pic{cross-v-axis} ,
      \\
      \pic{cross-u-axis} & = B^{-1} \pic{cross-h-axis} + B^{+1}\pic{cross-v-axis} .
    \end{aligned}
    \]
    
  \item
    If $C$ is a collection of $n$ circles 
    (i.e., a diagram without any crossings)
    having $2m$ intersections with the axis, 
    then we have the following circle evaluation formula: 
    \[
    \tag{C} \label{eq:SkeinC}
    \begin{aligned}
      \bracket{C} 
      & = (-A^2-A^{-2})^{n-1} \left(\frac{B^2+B^{-2}}{A^2+A^{-2}}\right)^{m-1} 
      \\
      \notag
      & = (-A^2-A^{-2})^{n-m} (-B^2-B^{-2})^{m-1} .
    \end{aligned}
    \]

  \end{itemize}
\end{definition}

\begin{remark}
  While the skein relations \eqref{eq:SkeinA} and \eqref{eq:SkeinB}
  are a natural ansatz, the circle evaluation formula \eqref{eq:SkeinC} 
  could seem somewhat arbitrary.  We should thus point out that, 
  if we want to achieve invariance, then \eqref{eq:SkeinA} and \eqref{eq:SkeinB} 
  imply \eqref{eq:SkeinC} up to a constant factor.
  We choose our normalization such that the unknot $\cpic{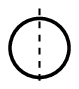}$
  (where $n=m=1$) is mapped to $\bracket{\cpic{trivial-axis}} = 1$.
\end{remark}

There is a natural family of Reidemeister moves respecting the axis, 
as recalled in \sref{sub:SymmetricMoves}.  The crucial observation 
is that the refined bracket is indeed invariant:

\begin{lemma}[regular invariance]
  The two-variable bracket $\bracket{D} \in \Z(A,B)$ is invariant 
  under regular Reidemeister moves respecting the axis.
  R1-moves off the axis contribute a factor $-A^{\pm3}$, 
  whereas S1-moves on the axis contribute a factor $-B^{\pm3}$.
\end{lemma}

\begin{remark}
  Of course, in every construction of link invariants 
  one can artificially introduce new variables.
  Usually the invariance under Reidemeister moves
  enforces certain relations and eliminates superfluous variables.
  It is thus quite remarkable that the variables 
  $A$ and $B$ remain free, and moreover, carry 
  geometric information as we shall see.
\end{remark}

\subsection{A refined Jones polynomial} 

In order to obtain full invariance we normalize the two-variable
bracket polynomial $\bracket{D}$ with respect to the writhe.  
To this end we consider the set $\odiagrams$ of oriented diagrams and
define the \emph{$A$-writhe} $\alpha(D)$ and the \emph{$B$-writhe} $\beta(D)$
to be the sum of crossing signs off and on the axis, respectively.
This ensures full invariance:

\begin{theorem}[refined Jones polynomial]
  The map $W \colon \odiagrams \to \Z(A,B)$ defined by 
  \[
  W(D) := \bracket{D} \cdot (-A^{-3})^{\alpha(D)} \cdot (-B^{-3})^{\beta(D)} 
  \]
  is invariant under all Reidemeister moves respecting the axis
  (displayed in \sref{sub:SymmetricMoves}).
\end{theorem}

\begin{notation}
  We shall adopt the common notation $A^2 = t^{-\onehalf}$ and $B^2 = s^{-\onehalf}$.
  Instead of $W(D)$ we also write $W_D$ or $W_D(s,t)$ if we wish 
  to emphasize or specialize the variables.
\end{notation}

The following properties generalize those of the Jones polynomial:

\begin{proposition} \label{prop:GeneralProperties}
  The invariant $W \colon \odiagrams \to \Z(s^\onehalf,t^\onehalf)$ 
  enjoys the following properties:
  \begin{enumerate}
  \item
    $W_D$ is insensitive to reversing the orientation of all components of $D$.
  \item
    $W_D$ is invariant under mutation, flypes, and rotation about the axis.
  \item
    If $D \sharp D'$ is a connected sum along the axis,
    then $W_{D \sharp D'} = W_D \cdot W_{D'}$.
  \item
    If $D^*$ is the mirror image of $D$, then $W_{D^*}(s,t) = W_D(s^{-1},t^{-1})$.
  \item
    If $D$ is a symmetric diagram, then $W_D(s,t)$ is symmetric in $t \leftrightarrow t^{-1}$.
  \item 
    If $D$ is a symmetric union link diagram, then $W_D$ is insensitive to
    reversing the orientation of any of the components of $D$.
  \end{enumerate}
\end{proposition}

\subsection{Symmetric unions}

In the special case of symmetric union diagrams,
the practical calculation of $W$-polynomials is 
most easily carried out via the following algorithm:

\begin{proposition}[recursive calculation via skein relations] 
  Consider a symmetric union diagram $D$ with $n$ components.  
  If $D$ has no crossings on the axis then
  \begin{equation} 
    W_D(s,t) = \left(\frac{s^\onehalf+s^{-\onehalf}}{t^\onehalf+t^{-\onehalf}}\right)^{n-1} V_L(t) ,
  \end{equation}
  where $V_L(t)$ is the Jones-polynomial of the link $L$ represented by $D$.  

  If $D$ has crossings on the axis, then we can apply the following recursion formulae:
  \renewcommand{\pic}[1]{W{\left(\raisebox{-0.9ex}{\includegraphics[height=3ex]{#1}}\right)}}
  \begin{align}
    \pic{cross-o-axis} &= -s^{+\onehalf} \pic{cross-h-axis} - s^{+1} \pic{cross-v-axis} ,
    \\
    \pic{cross-u-axis} &= -s^{-\onehalf} \pic{cross-h-axis} - s^{-1} \pic{cross-v-axis} .
  \end{align}
\end{proposition}

These rules allow for a recursive calculation of $W(D)$ for every symmetric union $D$. 
Notice that $W(D)$ is independent of orientations according 
to Proposition \ref{prop:GeneralProperties}(6).

We emphasize that $W(D)$ of an arbitrary diagram $D$ 
will in general \emph{not} be a polynomial: 
by construction $W(D) \in \Z(s^\onehalf,t^\onehalf)$ is usually a fraction 
and cannot be expected to lie in the subring $\Z[s^{\pm\onehalf},t^{\pm\onehalf}]$.
This miracle happens, however, for symmetric union diagrams:

\begin{proposition}[integrality] \label{prop:Integrality}
  If $D$ is a symmetric union knot diagram, then 
  $W_D$ is a Laurent polynomial in $s$ and $t$.
  More generally, if $D$ is a symmetric union diagram with $n$ components, then
  $W_D \in \Z[s^{\pm1},t^{\pm1}] \cdot (s^{\onehalf}+s^{-\onehalf})^{n-1}$.
\end{proposition}

\begin{remark}
  The integrality of $W_D$ is a truly remarkable property of symmetric unions.
  The fact that the denominator disappears for symmetric unions 
  was rather unexpected, and sparked off an independent investigation,
  whose results are presented in \cite{Eisermann:2009}.  
  The integrality of $W_D(s,t)$ now follows from 
  a more general integrality theorem \cite[Theorem 1]{Eisermann:2009},
  which is interesting in its own right:
  for every $n$-component ribbon link the Jones polynomial $V(L)$
  is divisible by the Jones polynomial $V(\bigcirc^n)$ of the trivial link.
\end{remark}

The following special values in $t$ correspond to those of the Jones polynomial:

\begin{proposition}[special values in $t$]
  \label{prop:SpecialValues:t}
  If $D$ is a symmetric union link diagram with $n$ components,
  then $W_D(s,\xi) = (-s^{\onehalf}-s^{-\onehalf})^{n-1}$ 
  for each $\xi \in \{ 1, \, \pm i, \, e^{\pm 2 i \pi / 3} \}$,
  and $\frac{\partial W_D}{\partial t}(s,1) = 0$.
  In other words, $W_D - (-s^{\onehalf}-s^{-\onehalf})^{n-1}$
  is divisible by $(t-1)^2 (t^2+1) (t^2+t+1)$.
\end{proposition}

The following special values in $s$ nicely reflect the symmetry:

\begin{proposition}[special values in $s$]
  Suppose that a knot $K$ can be represented by a symmetric union diagram $D$ 
  with partial knots $K_\pm$.  Then the following properties hold:
  \begin{enumerate}
  \item
    Mapping $s \mapsto t$ yields $W_D(t,t) = V_K(t)$, the Jones polynomial of $K$
  \item
    Mapping $s \mapsto -1$ yields a symmetric product $W_D(-1,t) = V_{K_-}(t) \cdot V_{K_+}(t)$.
  \end{enumerate}
  In particular, both specialization together imply 
  $W_D(-1,-1) = \det(K) = \det(K_-) \cdot \det(K_+)$.
\end{proposition}

\begin{remark}
  Finding a symmetric union representation $D$ for a ribbon knot $K$
  introduces precious extra structure that can be used to refine 
  the Jones polynomial $V_K(t)$ to a two-variable polynomial $W_D(s,t)$.
  In this sense we can interpret $W_D(s,t)$ as a ``lifting'' 
  of $V_K(t)$ to this richer structure. The specialization
  $s \mapsto t$ forgets the extra information and projects
  back to the initial Jones polynomial.  

  The product formula $W_D(-1,t) = V_{K_-}(t) \cdot V_{K_+}(t)$ is particularly intriguing.
  Recall that for every ribbon (or slice) knot $K$, 
  the Alexander-Conway polynomial is a symmetric product 
  $\Delta_K(t) = f(t) \cdot f(t^{-1})$ for some polynomial $f \in \Z[t^{\pm 1}]$.
  The preceding theorem says that such a symmetric product also appears 
  for the Jones polynomial $V_K(t)$, albeit indirectly via 
  the lifted two-variable polynomial $W_D(s,t)$.
\end{remark}

\begin{remark}
  We use the letter $W$ as a typographical reminder of the symmetry 
  that we wish to capture: $W$ is the symmetric union of two letters $V$,
  just as the $W$-polynomial is the combination of two $V$-polynomials.  
  (This analogy is even more complete in French, where $V$ 
  is pronounced ``v\'e'', while $W$ is pronounced ``double v\'e''.)  
\end{remark}

\subsection{Applications and examples}

In \cite{EisermannLamm:2007} we motivated the question
whether the two symmetric unions of \fref{fig:Knot-9_27}
could be symmetrically equivalent.  (In fact, $9_{27}$ is the first example
in an infinite family of two-bridge ribbon knots, see \sref{sub:TwoBridgeRibbonKnots}.)
Having the $W$-polynomial at hand, we can now answer this question in the negative:

\begin{example} \label{exm:knot-9_27}
  The symmetric union diagrams $D$ (left) and $D'$ (right) of \fref{fig:Knot-9_27}
  both represent the knot $9_{27}$.  The partial knot is $K_\pm = 5_2$ 
  in both cases, so this is no obstruction to symmetric equivalence
  (see \sref{sub:PartialKnots}).  Calculation of their $W$-polynomials yields:
  \begin{alignat*}{3}
    W_{D}(s,t) &=  1 + \;& s \cdot & g_1(t) - \;& s^2 \cdot & f(t) ,
    \\
    W_{D'}(s,t) &= 1 - & & g_1(t) + & s^{-1} \; \cdot & f(t) ,
  \end{alignat*}
  with
  \begin{align*}
    g_1(t) & = t^{-5}-3t^{-4}+6t^{-3}-9t^{-2}+11t^{-1}-12+11t-9t^2+6t^3-3t^4+t^5 ,
    \\
    f(t) & =  t^{-4}-2t^{-3}+3t^{-2}-4t^{-1}+4-4t+3t^2-2t^3+t^4 .
  \end{align*}
  This proves that $D$ and $D'$ are not equivalent 
  by symmetric Reidemeister moves.

  As an illustration, for both diagrams the specializations $s = -1$ and $s = t$ yield 
  \begin{align*}
    W(-1,t) & = (t-t^2+2t^3-t^4+t^5-t^6)(t^{-1}-t^{-2}+2t^{-3}-t^{-4}+t^{-5}-t^{-6}) , 
    \\ 
    W(t,t) & = t^{-4}-3t^{-3}+5t^{-2}-7t^{-1}+9-8t+7t^{2}-5t^{3}+3t^{4}-t^{5} .  
  \end{align*}
  
  Notice that $W(s,t)$ captures the symmetry, which is lost 
  when we pass to the Jones polynomial $V(t) = W(t,t)$.
  The latter does not seem to feature any special properties.
\end{example}

\begin{remark}
  Symmetric Reidemeister moves do not change the ribbon surface, 
  see Remark \ref{rem:MovesRespectSurface} below.
  Possibly the more profound difference between the 
  two symmetric union presentations $D$ and $D'$ of the knot $9_{27}$ is that 
  they define essentially distinct ribbon surfaces $S$ and $S'$ bounding the same knot $9_{27}$.
  To study this problem we would like to concoct an invariant $S \mapsto
  W_S(s,t)$ of (not necessarily symmetric) ribbon surfaces $S \subset \R^3$. 
  Ideally this would generalize our $W$-polynomial $W_D(s,t)$
  and likewise specialize to the Jones polynomial $V_K(t)$.
  In any case Figure \ref{fig:Knot-9_27} will provide a good test case
  to illustrate the strength of this extended invariant yet to be constructed.
\end{remark}

\subsection{Open questions}

Our construction works fine for symmetric unions, and we are 
convinced that this case is sufficiently important to merit investigation.
Ultimately, however, we are interested in ribbon knots.
Two possible paths are imaginable:

\begin{question}
  Can every ribbon knot be presented as a symmetric union?
\end{question}

Although this would be a very attractive presentation,
it seems rather unlikely.  

\begin{question}
  Is there a natural extension of the $W$-polynomial to ribbon knots?
\end{question}

This seems more plausible, but again 
such a construction is far from obvious.

The right setting to formulate these questions
is the following instance of ``knots with extra structure'',
where the vertical arrows are the obvious forgetful maps:
\[
\begin{CD}
  \left\{\begin{matrix}\text{symmetric}\\\text{unions}\end{matrix}\right\} 
  @>>> \left\{\begin{matrix}\text{ribbon knots +}\\\text{specific ribbon}\end{matrix}\right\} 
  @>>> \left\{\begin{matrix}\text{slice knots +}\\\text{specific slice}\end{matrix}\right\} 
  \\
  @VVV @VVV @VVV
  \\
  \left\{\begin{matrix}\text{symmetrizable}\\\text{ribbon knots}\end{matrix}\right\} 
  @>>> \{\text{ribbon knots}\}
  @>>> \{\text{slice knots}\}
\end{CD}
\]

Some natural questions are then:
Which ribbon knots are symmetrizable?
Which ribbons can be presented as symmetric unions?
Under which conditions is such a presentation unique?
(The analogous questions for the passage from slice to ribbon 
have already attracted much attention over the last 50 years.)

\begin{question}
  Can we construct an analogue of the $W$-polynomial
  for ribbon knots with a specified ribbon?  Does it extend
  the $W$-polynomial of symmetric unions, or do we have to pass 
  to a suitable quotient? 
\end{question}

\begin{question}
  Can one obtain in this way an obstruction for a knot to be ribbon? 
  Or an obstruction to being a symmetric union?
  (Although the $W$-polynomial captures the symmetry condition,
  it does not yet seem to provide such an obstruction.)
\end{question}

\begin{question}
  Are there similarly refined versions of the \textsc{Homflypt}
  and Kauffman polynomials?  Do we obtain equally nice properties?
\end{question}

\subsection{How this article is organized}

The article follows the program laid out in the introduction.  
Section \ref{sec:Definitions} expounds the necessary facts 
about symmetric diagrams (\sref{sub:SymmetricDiagrams}) and 
in particular symmetric unions (\sref{sub:SymmetricUnions}).
We then recall symmetric Reidemeister moves (\sref{sub:SymmetricMoves}) 
and sketch a symmetric Reidemeister theorem (\sref{sub:SymmetricReidemeisterTheorem}).
This is completed by a brief discussion of partial knots (\sref{sub:PartialKnots}) 
and Reidemeister moves respecting the axis (\sref{sub:AsymmetricMoves}).

Section \ref{sec:BracketConstruction} is devoted to the construction 
of the two-variable bracket (\sref{sub:BracketConstruction}) and 
its normalized version, the $W$-polynomial (\sref{sub:Normalization}).
In Section \ref{sec:GeneralProperties} we establish some general properties 
analogous to those of the Jones polynomial. 
Section \ref{sec:SymmetricUnionPolynomial} focuses on properties that are 
specific for symmetric union diagrams (\sref{sub:SymmetricUnionPolynomial}),
in particular integrality (\sref{sub:Integrality}) and special values 
in $t$ and $s$ (\sref{sub:SpecialValues:t}-\sref{sub:SpecialValues:s}).

Section \ref{sec:Examples} discusses examples and applications:
we compile a list of symmetric union diagrams and their $W$-polynomials
for all ribbon knots up to $10$ crossings (\sref{sub:SmallRibbonKnots})
and study two infinite families of symmetric union diagrams 
of two-bridge ribbon knots (\sref{sub:TwoBridgeRibbonKnots}).

\subsection{Acknowledgements}

The authors would like to thank Adam Sikora for helpful discussions
in Warsaw 2007.  This work was begun in the winter term 2006/2007 when 
the first author was on a sabbatical funded by a research contract
\textit{d\'el\'egation aupr\`es du CNRS}, whose support is gratefully acknowledged.


\section{Symmetric diagrams and symmetric equivalence} \label{sec:Definitions}

In this section we discuss symmetric diagrams and symmetric Reidemeister moves.  
Since we will use them in the next section to define our two-variable refinement 
of the Jones polynomial, we wish to prepare the stage in sufficient detail. 
It will turn out that our construction of the $W$-polynomial 
applies not only to symmetric unions but more generally 
to diagrams that are transverse to some fixed axis.  
In fact, the skein relations that we employ will destroy
the symmetry and thus make this generalization necessary.

\subsection{Symmetric diagrams} \label{sub:SymmetricDiagrams}

We consider the plane $\R^2$ with the reflection
$\rho\colon \R^2 \to \R^2$ defined by $(x,y) \mapsto (-x,y)$.
The map $\rho$ reverses the orientation of $\R^2$ and its 
fixed-point set is the vertical axis $\{0\} \times \R$.

\begin{definition}
  A link diagram $D \subset \R^2$ is \emph{symmetric} 
  if it satisfies $\rho(D) = D$ except for crossings 
  on the axis, which are necessarily reversed.
  By convention we consider two diagrams $D$ and $D'$ as identical 
  if they differ only by an orientation preserving diffeomorphism 
  $h \colon \R^2 \isoto \R^2$ respecting the symmetry, in the sense 
  that $h(D) = D'$ with $h \circ \rho = \rho \circ h$.
\end{definition}

\begin{figure}[hbtp]
  \centering
  \hfill
  \subfigure[the knot $6_1$]{\includegraphics[scale=1.0]{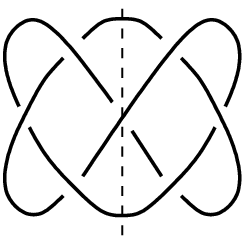}}
  \hfill
  \subfigure[the trefoil knot]{\includegraphics[scale=1.0]{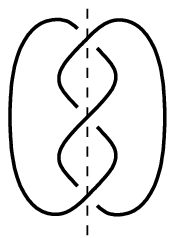}}
  \hfill
  \subfigure[the Hopf link]{\includegraphics[scale=1.0]{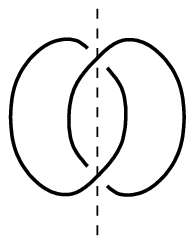}}
  \hfill{}
  \caption{Three types of symmetric diagrams}
  \label{fig:SymmetricDiagrams}
\end{figure} 

\begin{remark} \label{rem:ThreeComponentTypes}
  Each component $C$ of a symmetric diagram is of one of three types:
  \begin{enumerate}
  \item[(a)]
    The reflection $\rho$ maps $C$ to itself reversing 
    the orientation, as in \fref{fig:SymmetricDiagrams}a.
  \item[(b)]
    The reflection $\rho$ maps $C$ to itself preserving 
    the orientation, as in \fref{fig:SymmetricDiagrams}b.
  \item[(c)]
    The reflection $\rho$ maps $C$ to another component $\rho(C) \ne C$,
    as in \fref{fig:SymmetricDiagrams}c.
  \end{enumerate}
  
  Each component $C$ can traverse the axis in an arbitrary number of crossings.
  In cases (a) and (b) these are pure crossings where the component $C$ 
  crosses itself, while in case (c) they are mixed crossings 
  between the component $C$ and its symmetric partner $\rho(C)$.
  
  Moreover, the component $C$ can traverse the axis 
  without crossing any other strand;  assuming smoothness 
  this is necessarily a perpendicular traversal.
  In case (a) there are precisely two traversals of this kind, 
  while in cases (b) and (c) there are none.
\end{remark}

\subsection{Symmetric unions} \label{sub:SymmetricUnions}

In view of the preceding discussion of symmetric diagrams, 
we single out the case of interest to us here:

\begin{definition}
  We say that a link diagram $D$ is a \emph{symmetric union} 
  if it is symmetric, $\rho(D)=D$, and each component is of type (a).
  This means that each component perpendicularly traverses the axis 
  in exactly two points that are not crossings, and upon reflection 
  it is mapped to itself reversing the orientation.
\end{definition}

While symmetric diagrams in general are already interesting, 
symmetric unions feature even more remarkable properties.
Most notably they are ribbon links:

\begin{definition}
  Let $\Sigma$ be a compact surface, not necessarily connected nor orientable.
  A \emph{ribbon surface} is a smooth immersion $f \colon \Sigma \looparrowright \R^3$ 
  whose only singularities are ribbon singularities according to 
  the local model shown in \fref{fig:RibbonSingularity}a:
  the surface intersects itself in an interval $A$, whose
  preimage $f^{-1}(A)$ consists of one interval in the interior of $\Sigma$
  and a second, properly embedded interval, running from boundary to boundary.
\end{definition}

\begin{figure}[hbtp]
  \centering
  \hfill
  \subfigure[Local model of a ribbon singularity]
  {\figbox{25ex}{30ex}{\includegraphics[width=30ex]{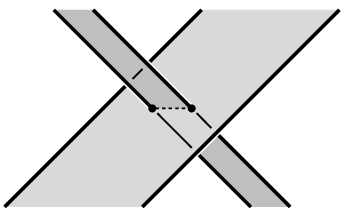}}}
  \hfill
  \subfigure[The knot $8_{20}$ bounding a disk with two ribbon singularities (dotted lines)]
  {\figbox{25ex}{30ex}{\includegraphics[width=20ex]{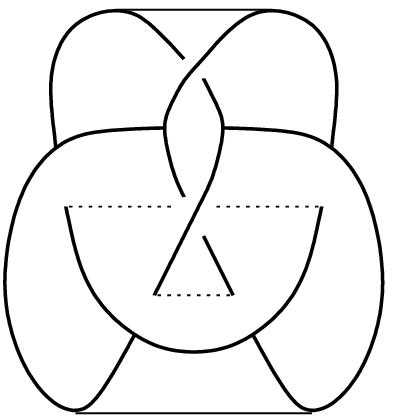}}}
  \hfill{}
  \caption{An immersed disk with ribbon singularities}
  \label{fig:RibbonSingularity}
\end{figure} 

\begin{definition}
  A link $L \subset \R^3$ is said to be a \emph{ribbon link}
  if it bounds a ribbon surface consisting of disks.
  (\fref{fig:RibbonSingularity}b shows an example.)
\end{definition}

\begin{proposition} \label{prop:SymmetricRibbonDisks}
  Every symmetric union diagram $D$ represents a ribbon link.
\end{proposition}

\begin{proof}
  The essential idea can be seen in \fref{fig:RibbonSingularity}b;
  the following proof simply formalizes this construction.
  We equip the disk $\D^2 = \{ z \in \R^2 \mid |z| \le 1 \}$ with 
  the induced action of the reflection $\rho \colon (x,y) \mapsto (-x,y)$, 
  and extend this action to $\Sigma = \{1,\dots,n\} \times \D^2$.
  The symmetric diagram $D$ can be parametrized by an equivariant plane curve 
  $g \colon \partial \Sigma \to \R^2$, satisfying $g \circ \rho = \rho \circ g$.
  We realize the associated link by a suitable lifting 
  $\tilde{g} \colon \partial \Sigma \to \R^3$ that projects to 
  $g = p \circ \tilde{g}$ via $p \colon \R^3 \to \R^2$, $(x,y,z) \mapsto (x,y)$.  
  We denote by $\tilde\rho \colon \R^3 \to \R^3$ the reflection
  $\tilde\rho \colon (x,y,z) \mapsto (-x,y,z)$.
  We can achieve $\tilde{g} \circ \rho = \tilde\rho \circ \tilde{g}$ except 
  in an arbitrarily small neighbourhood of the reflection plane 
  $\{0\} \times \R^2$ to allow for twists. 
  The map $\tilde{g}$ can be extended to a map $f \colon \Sigma \to \R^3$
  by connecting symmetric points by a straight line:
  \[
  f\bigl( (1-t) \cdot s + t \cdot \rho(s) \bigr) 
  = (1-t) \cdot \tilde{g}(s) + t \cdot \tilde{g}(\rho(s))
  \]
  for each $s \in \partial \Sigma$ and $t \in [0,1]$.  
  If we choose the lifting $\tilde{g}$ of $g$ generically, 
  then $f$ will be the desired ribbon immersion.
\end{proof}

An analogous construction can be carried out for an arbitrary symmetric diagram:

\begin{proposition} \label{prop:SymmetricRibbonSurface}
  Every symmetric diagram $D$ represents a link $L$ 
  together with a ribbon surface $f \colon \Sigma \looparrowright \R^3$ of the following type:
  \begin{enumerate}
  \item[(a)]
    Each component of type (a) bounds an immersed disk.
  \item[(b)]
    Each component of type (b) bounds an immersed M\"obius band.
  \item[(c)]
    Each pair of components of type (c) bounds an immersed annulus.
    \qed  
  \end{enumerate}
\end{proposition}

Let us add a remark that will be useful in \sref{sub:Integrality}.
Each disk contributes an Euler characteristic $1$ 
whereas annuli and M\"obius bands contribute $0$.
We conclude that $L$ bounds a ribbon surface of 
Euler characteristic $\chi(\Sigma) = n$, 
where $n$ is the number of components of type (a).
Moreover, since $D$ is symmetric, it perpendicularly traverses the axis 
precisely $2n$ times, twice for each component of type (a).

\subsection{Symmetric Reidemeister moves} \label{sub:SymmetricMoves}

Symmetric diagrams naturally lead to the 
following notion of symmetric Reidemeister moves:

\begin{definition}
  We consider a knot or link diagram that is symmetric with respect to 
  the reflection $\rho$ along the axis $\{0\} \times \R$.

  A \emph{symmetric Reidemeister move off the axis} is an ordinary
  Reidemeister move as depicted in \fref{fig:Rmoves} carried out 
  simultaneously with its mirror-symmetric counterpart.
  
  A \emph{symmetric Reidemeister move on the axis}
  is either an ordinary Reidemeister move (S1--S3) 
  or a generalized Reidemeister move (S2$\pm$ or S4)
  as depicted in \fref{fig:Smoves}.

  Subsuming both cases, a \emph{symmetric Reidemeister move}
  is one of the previous two types, either on or off the axis.
\end{definition}

\begin{figure}[hbtp]
  \centering
  \includegraphics[scale=0.8]{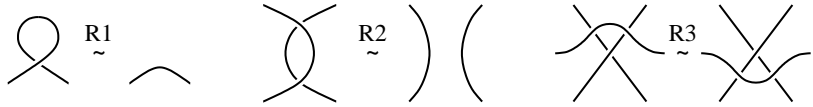}
  \caption{The classical Reidemeister moves (off the axis)}
  \label{fig:Rmoves}
\end{figure} 

\begin{figure}[hbtp]
  \centering
  \includegraphics[scale=0.8]{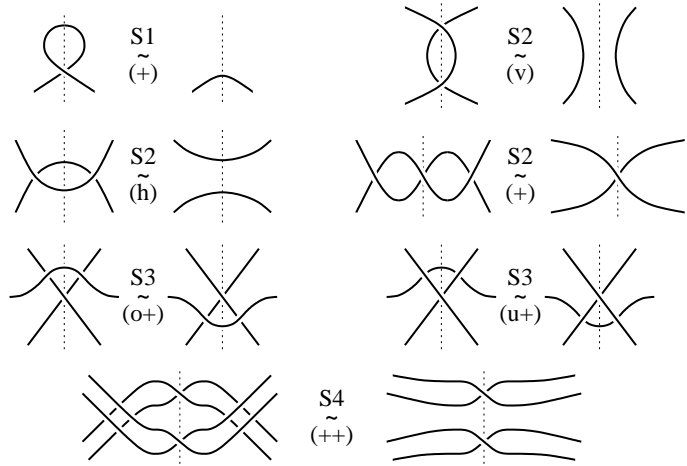}
  \caption{Symmetric Reidemeister moves on the axis}
  \label{fig:Smoves}
\end{figure} 

\begin{remark} 
  We usually try to take advantage of symmetries
  in order to reduce the number of local moves.
  By convention the axis is not oriented, which means that we can 
  turn all local pictures in \fref{fig:Smoves} upside-down.
  This adds one variant for each S1-, S2-, and S4-move shown here;
  the four S3-moves are each invariant under this rotation.
  We can also reflect each local picture along the axis,
  which exchanges the pairs S1$\pm$, S2$\pm$, S3o$\pm$, S3u$\pm$.
  Finally, we can rotate about the axis, which exchanges S3o and S3u.
  The S4-move, finally, comes in four variants, obtained by 
  changing the over- and under-crossings on the axis.
\end{remark}

\begin{remark}
  The S1 and S2v moves are special cases of a flype move 
  along the axis, as depicted in \fref{fig:FlypeAlongAxis}.
  The introduction of such flypes provides a strict generalization,
  because complex flypes along the axis can in general 
  not be generated by the above Reidemeister moves, 
  as observed in Remark \ref{rem:FlypePartialKnots} below.
  In particular, a half-turn of the entire diagram around the axis 
  can be realized by flypes, but not by symmetric Reidemeister moves.
\end{remark}

\begin{figure}[hbtp]
  \centering
  \includegraphics[scale=0.8]{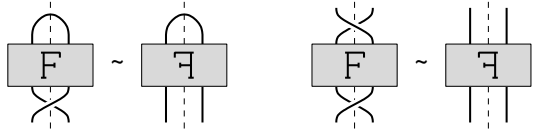}
  \caption{A vertical flype along the axis}
  \label{fig:FlypeAlongAxis}
\end{figure} 

\begin{remark} \label{rem:MovesRespectSurface}
  Symmetric Reidemeister moves as well as flypes
  preserve the ribbon surface constructed in
  Proposition \ref{prop:SymmetricRibbonSurface}: 
  every such move extends to an isotopy of the surface, 
  perhaps creating or deleting redundant ribbon singularities.
\end{remark}

\subsection{A  symmetric Reidemeister theorem} \label{sub:SymmetricReidemeisterTheorem}

In this article we shall consider the symmetric moves above
as \emph{defining} symmetric equivalence.  
Two natural questions are in order.  On the one hand 
one might wonder whether our list could be shortened.
This is not the case, in particular the somewhat unexpected 
moves S2$\pm$ and S4 are necessary in the sense that they cannot 
be generated by the other moves \cite[Thm.\,2.3]{EisermannLamm:2007}.

On the other hand one may ask whether our list is complete.
In order to make sense of this question and to derive a symmetric 
Reidemeister theorem, we wish to set up a correspondence
between symmetric Reidemeister moves of symmetric diagrams 
and symmetric isotopy of symmetric links in $\R^3$.

The na\"ive formulation, however, will not work because crossings 
on the axis inhibit strict symmetry: links realizing 
symmetric union diagrams are mirror-symmetric off the axis 
but rotationally symmetric close to the axis. 

One way to circumvent this difficulty is to represent each crossing
on the axis by a singularity $\ascr$ together with a sign that
specifies its resolution: $\ascr \mathbin{\smash{\overset{+}{\mapsto}}} \aocr$
resp.\ $\ascr \mathbin{\smash{\overset{-}{\mapsto}}} \aucr$.
This reformulation ensures that the (singular) link
is strictly mirror-symmetric.  The signs can be chosen
arbitrarily and encode the symmetry defect after resolution.

More formally, a singular link is an immersion
$f \colon \{1,\dots,n\} \times \S^1 \hookrightarrow \R^3$
whose only multiple points are non-degenerate double points.
We shall not distinguish between different parametrizations
and thus identify the immersion $f$ and its image $L$.
We can then consider singular links $L \subset \R^3$ 
satisfying the following conditions:
\begin{description}
\item[Transversality]
  $L$ is transverse to $E = \{0\} \times \R^2$, and each double point lies on $E$.
\item[Symmetry]
  $L$ is symmetric with respect to reflection along $E$.
\end{description}
For such links we have the obvious notion of isotopy, that is, 
a smooth family $(L_t)_{t \in [0,1]}$ such that each $L_t$ satisfies 
the above transversality and symmetry requirements.
If the singularities are equipped with signs, then these
signs are carried along the isotopy in the obvious way.

\begin{theorem}
  Consider two symmetric diagrams $D_0$ and $D_1$ and 
  the associated symmetric (singular) links $L_0$ and $L_1$.
  If the links $L_0$ and $L_1$ are symmetrically isotopic
  then the diagrams $D_0$ and $D_1$ are symmetrically equivalent.
\end{theorem}

\begin{proof}[Sketch of proof]
  We can put the isotopy $(L_t)_{t \in [0,1]}$ into generic position such that 
  for all but a finite number of parameters $0 < t_1 < \dots < t_k < 1$
  the link $L_t$ projects to a symmetric diagram.  In particular,
  the diagrams between two successive parameters $t_i$ and $t_{i+1}$
  differ only by an isotopy of the plane and are essentially the same.
  Moreover we can arrange that at each exceptional parameter $t_i$
  the modification is of the simplest possible type:

  Events off the axis:
  \begin{itemize}
  \item
    The projection of a tangent line degenerates to a point: R1 move.
  \item
    Two tangent lines co\"incide in projection: R2 move.
  \item
    The projection produces a triple point: R3 move.
  \end{itemize}

  Events on the axis:
  \begin{itemize}
  \item
    Two tangent lines co\"incide in projection: S2h move.
  \item
    The tangent lines of a singular point become collinear in projection: S2$\pm$ move.
  \item
    A strand crosses a singular point: S3 move.
  \item
    Two singular points cross: S4 move.
  \end{itemize}

  The details of this case distinction shall be omitted.
\end{proof}

\begin{remark}
  We emphasize that, in the above setting of symmetric isotopy,
  moves of type S1 and S2v cannot occur.  Such isotopies can
  be realized only by temporarily breaking the symmetry.
  Instead of further enlarging the notion of isotopy in order 
  to allow for the creation and deletion of singularities,
  we simply introduce S1 and S2v as additional moves.
  We usually even allow the more general flype moves
  depicted in \fref{fig:FlypeAlongAxis}.
\end{remark}

\subsection{Partial knots} \label{sub:PartialKnots}

We are particularly interested in \emph{symmetric union knot diagrams}, 
where we require the symmetric union diagram 
to represent a knot $K$, that is, a one-component link.
As mentioned in the introduction, a symmetric union diagram 
of $K$ looks like the connected sum $K_+ \sharp K_-$
of a knot $K_+$ and its mirror image $K_-$, 
with additional crossings inserted on the symmetry axis.  
The following construction makes this observation precise:

\begin{definition}
  For every symmetric union knot diagram $D$ 
  we can define partial diagrams $D_-$ and $D_+$ as follows: 
  first, we resolve each crossing on the axis by cutting it open 
  according to $\aocr \mapsto \avcr$ or $\aucr \mapsto \avcr$.
  The result is a connected sum, which can then
  be split by a final cut $\ahcr \mapsto \avcr$.
  We thus obtain two disjoint diagrams: 
  $D_-$ in the halfspace $H_- = \{ (x,y) \mid x<0 \}$,
  and $D_+$ in the halfspace $H_+ = \{ (x,y) \mid x>0 \}$.
  The knots $K_-$ and $K_+$ represented by $D_-$ and $D_+$,
  respectively, are called the \emph{partial knots} of $D$.
\end{definition}

\begin{proposition}
  For every union diagram $D$ the partial knots $K_-$ and $K_+$ 
  are invariant under symmetric Reidemeister moves.
\end{proposition}

\begin{proof}
  This is easily seen by a straightforward case-by-case verification.
\end{proof}

\begin{remark} \label{rem:FlypePartialKnots}
  Notice that the partial knots are in general not invariant 
  under flypes along the axis, depicted in \fref{fig:FlypeAlongAxis}.
  Such moves can change the partial knots from $K_- \sharp L_-$ 
  and $K_+ \sharp L_+$ to $K_- \sharp L_+$ and $K_+ \sharp L_-$.
\end{remark}

\begin{remark}
  The above construction can be used to define the notion
  of \emph{partial link} for symmetric diagrams that have 
  components of type (b) and (c), and at most one component of type (a). 
  If there are two or more components of type (a), 
  then there does not seem to be a natural notion of partial knot or link.
  (A \emph{partial tangle} can, however, be defined as above,
  up to a certain equivalence relation induced by braiding the ends;
  we will not make use of this generalization in the present article.)
\end{remark}

\subsection{Reidemeister moves respecting the axis} \label{sub:AsymmetricMoves}

As an unintentional side-effect, most of our arguments 
will work also for \emph{asymmetric} diagrams.
Our construction of the bracket polynomial 
in \sref{sec:BracketConstruction} even \emph{requires} 
asymmetric diagrams in intermediate computations,
because the resolution of crossings breaks the symmetry.
Before stating the construction and the invariance theorem 
for our bracket polynomial, we thus make the underlying 
diagrams and their Reidemeister moves explicit.

As before we equip the plane $\R^2$ with the axis $\{0\} \times \R$,
but unlike the symmetric case, the reflection $\rho$ will play no r\^ole here.
We consider link diagrams that are transverse to the axis, that is, 
wherever a strand intersects the axis it does so transversally.
For such a diagram we can then distinguish crossings 
\emph{on} the axis and crossings \emph{off} the axis.

\begin{definition}
  We denote by $\udiagrams$ the set of planar link diagrams 
  that are transverse to the axis $\{0\} \times \R$,
  but not necessarily symmetric.  We do not distinguish 
  between diagrams that differ by an orientation-preserving 
  diffeomorphism $h \colon \R^2 \isoto \R^2$ fixing the axis setwise.
  A \emph{Reidemeister move respecting the axis}
  is a move of the following type:
  \begin{itemize}
  \item
    A Reidemeister move (R1, R2, R3) off the axis as depicted in \fref{fig:Rmoves}.
  \item
    A Reidemeister move (S1, S2, S3, S4) on the axis, 
    as depicted in \fref{fig:Smoves}.
  \end{itemize}
\end{definition}

The advantage of this formulation is that it applies to all diagrams,
symmetric or not.  For symmetric diagrams, both notions of equivalence co\"incide:

\begin{proposition}
  Two symmetric diagrams are equivalent under symmetric Reidemeister moves
  if and only if they are equivalent under Reidemeister moves respecting the axis.
\end{proposition}

\begin{proof}
  ``$\Rightarrow$'' 
  Each symmetric R-move is the composition of two asymmetric R-moves.

  ``$\Leftarrow$''
  Suppose that we can transform a symmetric diagram $D$ into 
  another symmetric diagram $D'$ by a sequence of R-moves and S-moves. 
  Since R-moves may be carried out asymmetrically,
  the symmetry of intermediate diagrams is lost.
  Nevertheless, the isotopy types of the tangles left and right 
  of the axis remain mutually mirror-symmetric, since S-moves
  preserve this symmetry.  We can thus forget the 
  given R-moves on the left-hand side of the axis, say.
  Each time we carry out an R-move on the right-hand side,
  we simultaneously perform its mirror image on the left-hand side.
  This defines a symmetric equivalence from $D$ to $D'$.
\end{proof}

\begin{remark}
  As before we can define the partial diagrams $D_-$ and $D_+$
  of a diagram $D$, provided that $D$ perpendicularly traverses
  the axis in either two points or no points at all.  The partial links
  $L_-$ and $L_+$ are invariant under Reidemeister moves respecting the axis.
\end{remark}


\section{Constructing the two-variable $W$-polynomial} \label{sec:BracketConstruction}

\subsection{Constructing the two-variable bracket polynomial} \label{sub:BracketConstruction}

We consider the set $\udiagrams$ of unoriented planar link diagrams that 
are transverse to the axis $\{0\} \times \R$ but not necessarily symmetric.
We can then define the bracket $\bracket{\cdot} \colon \udiagrams \to \Z(A,B)$
as in Definition \ref{def:TwoVariableBracket}.

\begin{lemma} \label{lem:BracketInvariance}
  The polynomial $\bracket{D}$ associated to a link diagram $D$ is invariant under 
  R2- and R3-moves off the axis as well as S2-, S3-, and S4-moves on the axis.
  It is not invariant under R1- nor S1-moves, but its behaviour is well-controlled: 
  we have
  \begin{gather}
    \label{eq:writhe}
    \renewcommand{\pic}[1]{\bracket{\raisebox{-1.4ex}{\includegraphics[height=4ex]{#1}}}}
    \pic{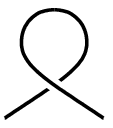} = (-A^3) \pic{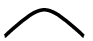} 
    \qquad \text{and} \qquad
    \pic{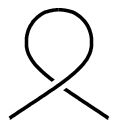} = (-A^{-3}) \pic{writhe0} ,
    \\
    \renewcommand{\pic}[1]{\bracket{\raisebox{-1.6ex}{\includegraphics[height=4.3ex]{#1}}}}
    \pic{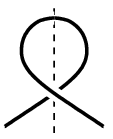}  = (-B^3) \pic{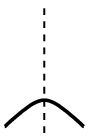}
    \qquad \text{and} \qquad
    \pic{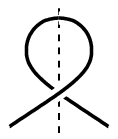} = (-B^{-3}) \pic{writhe0axis} .
  \end{gather}
\end{lemma}

\begin{proof}
  The proof consists of a case-by-case verification of 
  the stated Reidemeister moves.  It parallels Kauffman's 
  proof for his bracket polynomial, and is only 
  somewhat complicated here by a greater number of moves.

  Let us begin by noting two consequences of the circle evaluation formula \eqref{eq:SkeinC}:
  \begin{itemize}
  \item
    A circle \emph{off} the axis contributes a factor $(-A^2-A^{-2})$. 
  \item
    A circle \emph{on} the axis contributes a factor $(-B^2-B^{-2})$.
  \end{itemize}

  As a consequence, for Reidemeister moves of type R1$(+)$ we find 
  \renewcommand{\pic}[1]{\bracket{\raisebox{-1.4ex}{\includegraphics[height=4ex]{#1}}}}
  \begin{equation}
    \label{eq:r1move}
    \pic{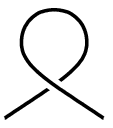}
    = A\pic{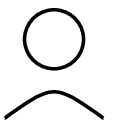} + A^{-1}\pic{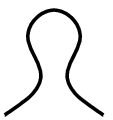} 
    = -A^3\pic{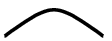} .
  \end{equation}
  The two summands contribute a factor $A(-A^2-A^{-2}) + A^{-1} = -A^3$, as claimed.
  The same calculation works for R1$(-)$, leading to a factor $-A^{-3}$.
  For S1-moves the calculation applies verbatim, replacing $A$ by $B$:
  \renewcommand{\pic}[1]{\bracket{\raisebox{-1.4ex}{\includegraphics[height=4.2ex]{#1}}}}
  \begin{equation}
    \label{eq:s1move}
    \pic{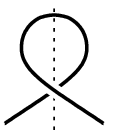}
    = B\pic{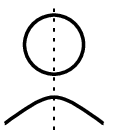} + B^{-1}\pic{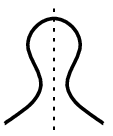} 
    = -B^3\pic{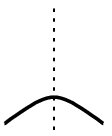} .
  \end{equation}

  Invariance under R2-moves is proven as usual, via the skein relation \eqref{eq:SkeinA}:
  \renewcommand{\pic}[1]{\bracket{\raisebox{-0.9ex}{\includegraphics[height=3ex]{#1}}}}
  \begin{equation}
    \label{eq:r2move}
    \pic{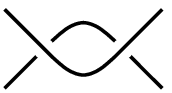}
    = A^{+2}\pic{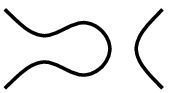} + A^{-2}\pic{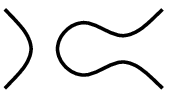} + \pic{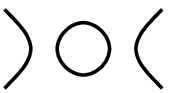} + \pic{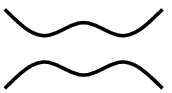} 
    = \pic{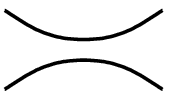}
  \end{equation}
  Here the first two summands cancel with the third, because 
  a circle \emph{off} the axis contributes a factor $(-A^2-A^{-2})$.

  Analogously, invariance under S2v-moves is proven via the skein relation \eqref{eq:SkeinB}:
  \renewcommand{\pic}[1]{\bracket{\raisebox{-3ex}{\includegraphics[width=3ex]{#1}}}}
  \begin{equation}
    \label{eq:s2vmove}
    \pic{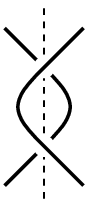}
    = B^{+2}\pic{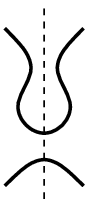} + B^{-2}\pic{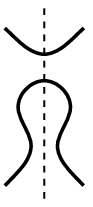} + \pic{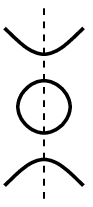} + \pic{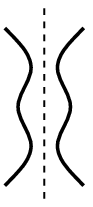}
    = \pic{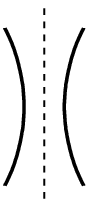}
  \end{equation}
  Here the first two summands cancel with the third, because 
  a circle \emph{on} the axis contributes a factor $(-B^2-B^{-2})$.

  Invariance under S2h-moves is proven as follows:
  \renewcommand{\pic}[1]{\bracket{\raisebox{-1.5ex}{\includegraphics[height=4.2ex]{#1}}}}
  \begin{equation}
    \label{eq:s2hmove}
    \pic{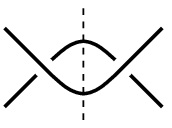}
    = A^{+2}\pic{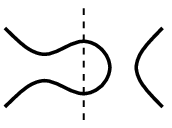} + A^{-2}\pic{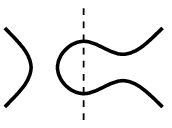} + \pic{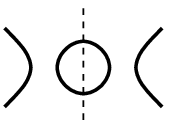} + \pic{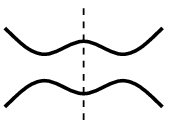}
    = \pic{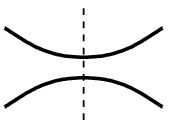}
  \end{equation}
  Here the first two summands cancel with the third,
  thanks to the judicious coupling of the variables $A$ and $B$, 
  as formulated in the circle evaluation \eqref{eq:SkeinC}:
  \begin{equation}
    \pic{s2hinvAA} = \pic{s2hinvBB} = \frac{B^2+B^{-2}}{A^2+A^{-2}} \pic{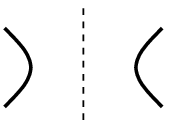}
  \end{equation}

  Invariance under the remaining moves will now be an easy consequence.  
  To begin with, S2h-invariance implies invariance under 
  the slightly more complicated move S2$\pm$:
  \renewcommand{\pic}[1]{\bracket{\raisebox{-1.5ex}{\includegraphics[height=4.2ex]{#1}}}}
  \begin{align}
    \label{eq:s2+move}
    \pic{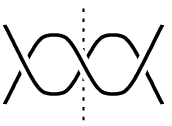}
    & = B\pic{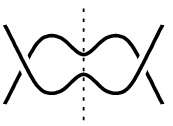} + B^{-1}\pic{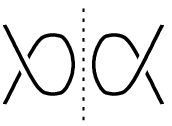} 
    \\ \notag
    & = B\pic{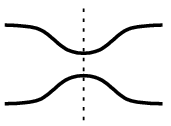} + B^{-1}\pic{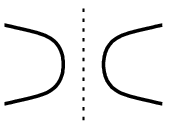} 
    = \pic{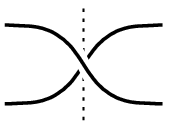}
  \end{align}
  Here the two $B$-summands are equal using S2h-invariance.
  For the $B^{-1}$-summand we carry out two opposite R1-moves, 
  so the factors $(-A^3)$ and $(-A^{-3})$ cancel each other.

  Invariance under R3-moves is proven as usual, via the skein relation \eqref{eq:SkeinA}:
  \renewcommand{\pic}[1]{\bracket{\raisebox{-2.2ex}{\includegraphics[height=5.6ex]{#1}}}}
  \begin{align}
    \label{eq:r3move}
    \pic{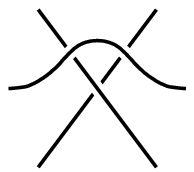}
    & = A\pic{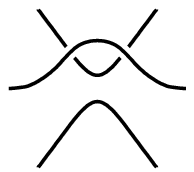} + A^{-1}\pic{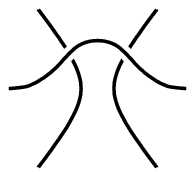}
    \\ \notag
    & = A\pic{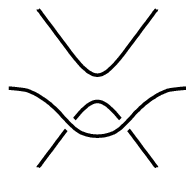} + A^{-1}\pic{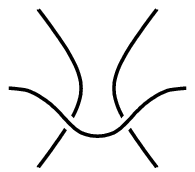}
    = \pic{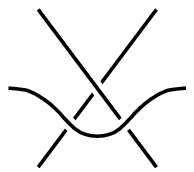}
  \end{align}
  Here the middle equality follows from R2-invariance, established above.
  Notice also that this R3-move comes in another variant: 
  if the middle crossing is changed to its opposite,
  then the coefficients $A$ and $A^{-1}$ are exchanged,
  and the desired equality is again verified.

  Analogously, invariance under S3-moves is proven
  via the skein relation \eqref{eq:SkeinB}:
  \begin{align}
    \label{eq:s3move}
    \pic{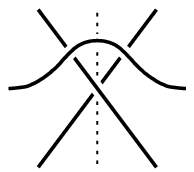}
    & = B\pic{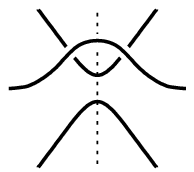} + B^{-1}\pic{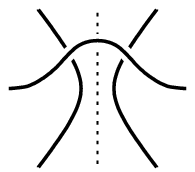}
    \\ \notag
    & = B\pic{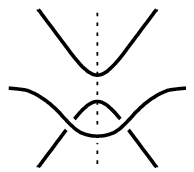} + B^{-1}\pic{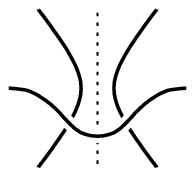}
    = \pic{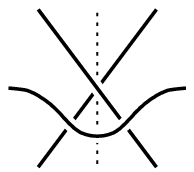}
  \end{align}
  Here the middle equality follows from S2h-invariance, established above.  
  This proves invariance under any R2v-move in the variant (o+).
  For the variant (o-) the middle crossing is changed to its opposite: 
  in the preceding equation the coefficients $B$ and $B^{-1}$ 
  are exchanged, and the desired equality is still verified.
  For the variants (u+) and (u-) the horizontal strand passes under 
  the two other strands, and the same argument still holds.

  Finally, invariance under S4-moves is again proven via the skein relation \eqref{eq:SkeinB}:
  \renewcommand{\pic}[1]{\bracket{\raisebox{-2.2ex}{\includegraphics[height=5.6ex]{#1}}}}
  \begin{align}
    \label{eq:s4move}
    \pic{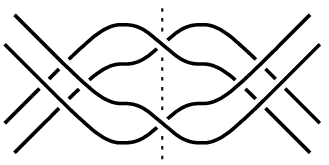}
    & = B\pic{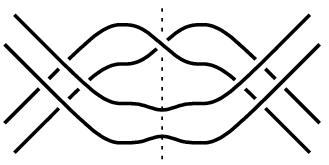} + B^{-1}\pic{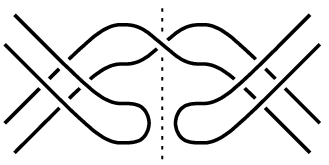}
    \\  \notag
    & = B\pic{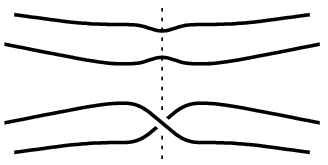} + B^{-1}\pic{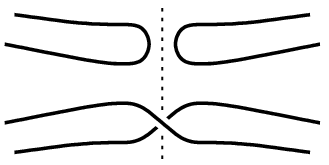}
    = \pic{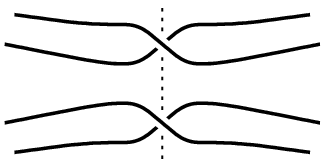}
  \end{align}
  The middle equality follows from S3- and R2-invariance, established above.  
  There are three more variants of S4-moves, obtained by changing 
  one or both of the middle crossings to their opposite.
  In each case the desired equality can be verified in the same way.
\end{proof}

\subsection{Normalizing with respect to the writhe} \label{sub:Normalization}

Given an oriented link diagram $D$, we can associate a sign to each crossing, 
according to the convention $\pcr \mapsto +1$ and $\ncr \mapsto -1$.
Let $\alpha(D)$ be the sum of crossing signs off the axis (called $A$-writhe),
and let $\beta(D)$ be the sum of crossing signs on the axis (called $B$-writhe).

\begin{figure}[hbtp]
  \centering
  \includegraphics{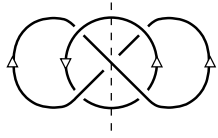}
  \caption{A diagram $D$ with $\alpha(D)=4$ and $\beta(D)=-1$}
  \label{fig:writhe}
\end{figure} 

\begin{definition}
  We define the normalized polynomial $W \colon \odiagrams \to \Z(A,B)$ to be
  \[
  W(D) := \bracket{D} \cdot (-A^3)^{-\alpha(D)} \cdot (-B^3)^{-\beta(D)}.
  \]
  This is called the \emph{$W$-polynomial} of the diagram $D$
  with respect to the given axis. 
\end{definition}

\begin{theorem}
  $W(D)$ is invariant under Reidemeister moves respecting the axis.
\end{theorem}

\begin{proof}
  The $A$-writhe $\alpha(D)$ does not change under regular Reidemeister moves.
  Since $\bracket{D}$ is also invariant under such moves, so is $W(D)$.
  \renewcommand{\pic}[1]{\raisebox{-0.5ex}{\includegraphics[height=2ex]{#1}}}
  An R1-move from $D = \pic{writhe+}$ to $D' = \pic{writhe0}$ 
  changes the $A$-writhe to $\alpha(D') = \alpha(D)-1$, 
  so that the factors in $W$ compensate according 
  to Lemma \ref{lem:BracketInvariance}.
  The same argument holds for S1-moves and the $B$-writhe.
\end{proof}

\begin{remark} \label{rem:SymmetricLinking}
  Consider a symmetric diagram $D$.  At first sight one would expect $\alpha(D) = 0$,
  so that no normalization has to be carried out for the variable $A$.
  Indeed, in almost all cases crossing signs cancel each other in symmetric pairs,
  but this fails where components of type (a) cross components 
  of type (b) or (c): according to Remark \ref{rem:ThreeComponentTypes}
  the reflection $\rho$ reverses the orientation of the former, 
  but preserves the orientation of the latter.  The signs in such 
  a symmetric pair of crossings are thus not opposite but identical.
  The simplest example of this kind is displayed in \fref{fig:writhe}, 
  showing in particular that $\alpha(D)$ can be non-zero.
\end{remark}

\subsection{Generalization to arbitrary surfaces} \label{sub:Generalization}

Our invariance arguments are local in nature, and thus 
immediately extend to any oriented connected surface $\Sigma$
equipped with a reflection, that is, an orientation-reversing
diffeomorphism $\rho \colon \Sigma \to \Sigma$ of order $2$.
Even though we do not have an immediate application for it,
this generalization seems natural and interesting enough
to warrant a brief sketch.
As before, we will call $\rho$ the \emph{reflection};
its fix-point set is a $1$-dimensional submanifold
which will be called the \emph{axis}.

\begin{figure}[hbtp]
  \centering
  \includegraphics[scale=0.8]{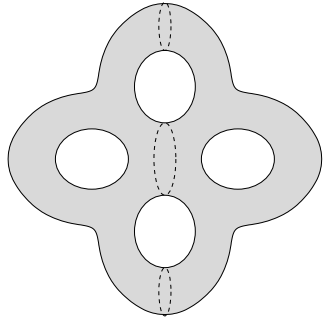}
  \caption{A surface $\Sigma$ with orientation-reversing involution $\rho$.
    The fixed axis is depicted as a dashed line.}
  \label{fig:Surface}
\end{figure} 

\begin{example}
  Such an object $(\Sigma,\rho)$ naturally arises for every complex manifold $\Sigma$ 
  of complex dimension $1$ (and real dimension $2$) equipped with a real structure, 
  that is, an antiholomorphic involution $\rho \colon \Sigma \to \Sigma$.    
  This includes the basic situation of the complex plane $\C$ or 
  the Riemann sphere $\CP^1$, with $\rho$ being complex conjugation.
  More generally, one can consider the zero-set $\Sigma \subset \C^2$ of a non-degenerate 
  real polynomial $P \in \R[z_1,z_2]$, or the zero-set $\Sigma \subset \CP^2$
  of a non-degenerate homogeneous polynomial $P \in \R[z_1,z_2,z_3]$,
  where the reflection $\rho$ is again given by complex conjugation.
\end{example}

\begin{remark}
  As in \sref{sub:SymmetricDiagrams}, a link diagram $D$ 
  on the surface $\Sigma$ is \emph{symmetric} if $\rho(D) = D$ 
  except for crossings on the axis, which are necessarily reversed.
  For symmetric diagrams we can consider symmetric Reidemeister 
  moves as in \sref{sub:SymmetricMoves} and establish a symmetric 
  Reidemeister theorem as in \sref{sub:SymmetricReidemeisterTheorem}.
  Partial tangles can be constructed as in \sref{sub:PartialKnots}
  and are again invariant; this is essentially a local property.
  In the absence of a convex structure, however, we cannot
  construct ribbon surfaces as in \sref{sub:SymmetricUnions}
  by joining opposite points.  More generally, a surface bounding $L$
  in $\Sigma\times\R$ exists if and only if the obvious obstruction 
  $[D] \in H_1(\Sigma)$ vanishes.
\end{remark}

\begin{remark}
  As before we can weaken the symmetry condition and
  consider only transverse diagrams under Reidemeister
  moves respecting the axis.  Here we assume a Morse function
  $h \colon \Sigma \to \R$ for which $0$ is a regular value,
  so that the axis $A = h^{-1}(0)$ decomposes $\Sigma$ into 
  two half-surfaces $\Sigma_- = \{ x \in \Sigma \mid h(x) < 0 \}$
  and $\Sigma_+ = \{ x \in \Sigma \mid h(x) > 0 \}$.
  
  We can then consider the set $\udiagrams(\Sigma)$ of 
  link diagrams on $\Sigma$ that are transverse to the axis.
  The skein relations \eqref{eq:SkeinA} and \eqref{eq:SkeinB}
  together with the circle evaluation formula \eqref{eq:SkeinC}
  define an invariant $\udiagrams(\Sigma) \to \Z(A,B)$ as before.
  This can be further refined in two ways.  Firstly,
  instead of one variable $B$ we can introduce separate
  variables $B_1,\dots,B_n$ for each connected component of the axis.
  Secondly, we can evaluate circles on the surface $\Sigma$ 
  according to their isotopy type.  
  The generalized construction essentially works as before.
\end{remark}


\section{General properties of the $W$-polynomial} \label{sec:GeneralProperties}

\subsection{Symmetries, connected sums, and mutations} \label{sub:Symmetries}

As before, we adopt the notation $A^2 = t^{-\onehalf}$ and $B^2 = s^{-\onehalf}$,
and instead of $W(D)$ we also write $W_D(s,t)$. 


\begin{proposition}
  $W_D$ is insensitive to reversing the orientation of all components of $D$.
\end{proposition}

\begin{proof}
  The bracket polynomial is independent of orientations,
  and the writhe does not change either: crossing signs 
  are invariant if we change \emph{all} orientations.
\end{proof}

\begin{proposition}
  The $W$-polynomial enjoys the following properties:
  \begin{enumerate}
  \item
    $W_D$ is invariant under mutation, flypes, and rotation about the axis.
  \item
    If $D \sharp D'$ is a connected sum along the axis,
    then $W_{D \sharp D'} = W_D \cdot W_{D'}$.
  \item
    If $D^*$ is the mirror image of $D$, then $W_{D^*}(s,t) = W_D(s^{-1},t^{-1})$.
  \item
    If $D$ is symmetric, then $W_D(s,t)$ is symmetric in $t \leftrightarrow t^{-1}$.
  \end{enumerate}
\end{proposition}

\begin{figure}[hbtp]
  \centering
  \subfigure[connected sum \label{subfig:ConnectedSum}]{\figbox{24ex}{28ex}{\includegraphics[scale=0.8]{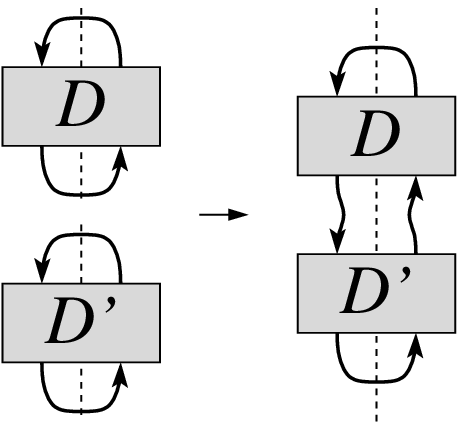}}}
  \qquad
  \subfigure[mutation \label{subfig:Mutation}]{\figbox{24ex}{28ex}{\includegraphics[scale=0.8]{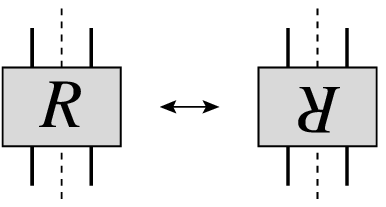}}} 
  \caption{Connected sum and mutation along the axis}
  \label{fig:ConnectedSumMutation}
\end{figure} 

\begin{proof}
  In each case the proof is by induction 
  on the number of crossings of $D$: 
  the assertion is clear when $D$ has no crossings
  and is propagated by the skein relations.
\end{proof}

Flypes and mutations along the axis are depicted in Figures 
\ref{fig:FlypeAlongAxis} and \ref{fig:ConnectedSumMutation}.
Such moves leave the $W$-polynomial invariant but
can change the partial knots, namely from $K_- \sharp L_-$ 
and $K_+ \sharp L_+$ to $K_- \sharp L_+$ and $K_+ \sharp L_-$.
For a discussion of connected sums see \cite{EisermannLamm:2007}:
there are different ways of forming a connected sum, 
but they are related by mutations.

\begin{remark}
  We point out the subtlety that there are two variants of mutation (\fref{fig:Mutations}): rotation and flipping. 
  (The combination of both variants yields a flip along a perpendicular axis and is not depicted here.)
  It is easy to see that both variants are equivalent: we can deduce a flip from 
  two rotations and some Reidemeister moves (\fref{fig:RotationInducesFlip}),
  and conversely a rotation from two flips (\fref{fig:FlipInducesRotation}).
  In our setting of diagrams with respect to a fixed axis, 
  this equivalence still holds for both variants of mutation on the axis,
  using Reidemeister moves and flypes respecting the axis.
  It is thus enough to consider the mutation depicted in \fref{subfig:Mutation}.
  The assertion of the theorem holds for mutations on and off the axis.
\end{remark}

\begin{figure}[hbtp]
  \centering
  \includegraphics[scale=0.8]{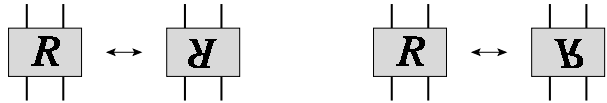}
  \caption{Two types of mutation: rotation and flipping}
  \label{fig:Mutations}
\end{figure} 

\begin{figure}[hbtp]
  \centering
  \includegraphics[scale=0.8]{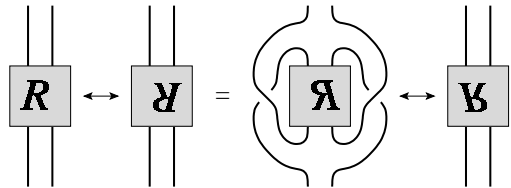}
  \caption{Deducing a flip from two rotations}
  \label{fig:RotationInducesFlip}
\end{figure} 

\begin{figure}[hbtp]
  \centering
  \includegraphics[scale=0.8]{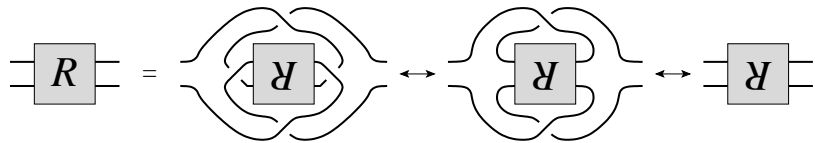}
  \caption{Deducing a rotation from two flips}
  \label{fig:FlipInducesRotation}
\end{figure}

\begin{example}
  The Kinoshita-Terasaka knot can be presented as a symmetric union 
  (with trivial partial knots) as in \fref{fig:KTC} on the left.  
  On the right you see a mutation, the Conway knot, 
  where the right half is flipped.  
  Both knots thus share the same $W$-polynomial.
\end{example}

\begin{figure}[hbtp]
  \includegraphics[height=28ex]{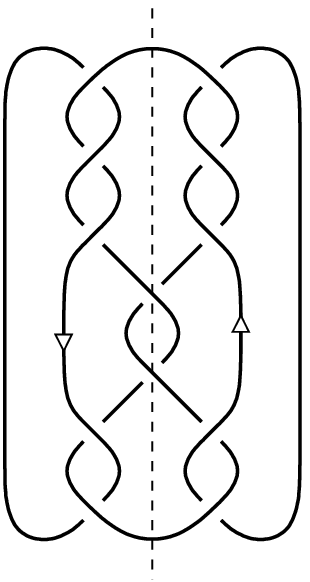}
  \qquad \qquad 
  \includegraphics[height=28ex]{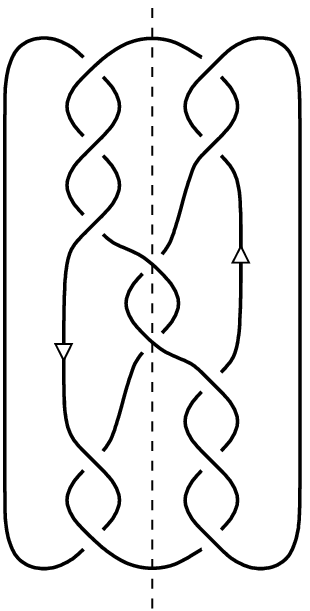}
  \caption{The Kinoshita-Terasaka knot (left) and the Conway knot
    (right) are mutations of one another.}
  \label{fig:KTC}
\end{figure}

\subsection{Oriented skein relations} \label{sub:OrientedSkeinRelations}

The following observation can be useful to simplify calculations,
by relating $W_D$ to the Jones polynomial in an important special case:

\begin{proposition} \label{prop:NoCrossingsOnAxis}
  Let $D$ be a diagram representing a link $L$.
  If $D$ has no crossings on the axis and perpendicularly 
  traverses the axis in $2m$ points, then 
  \[
  W_D(s,t) = \left(\frac{s^\onehalf+s^{-\onehalf}}{t^\onehalf+t^{-\onehalf}}\right)^{m-1} V_L(t) .
  \]
\end{proposition}

\begin{proof}
  The claim follows by induction on the number $c$ of crossings off the axis.
  If $c = 0$ then we simply have the circle evaluation formula \eqref{eq:SkeinC}.
  If $c \ge 1$ then we can resolve one crossing off the axis and apply 
  the skein relation \eqref{eq:SkeinA} on both sides of the equation.
\end{proof}

\begin{remark}
  The invariant $W \colon \odiagrams \to \Z(s^\onehalf,t^\onehalf)$
  satisfies some familiar skein relations:
  \renewcommand{\pic}[1]{W{\left(\raisebox{-0.9ex}{\includegraphics[height=3ex]{#1}}\right)}}
  \begin{align}
    t^{-1} \pic{skein+} - t^{+1} \pic{skein-} & = (t^\onehalf-t^{-\onehalf}) \pic{skein0}
    \\
    s^{-1} \pic{skein+axis} - s^{+1} \pic{skein-axis} & = (s^\onehalf-s^{-\onehalf}) \pic{skein0axis}
    \\
    \pic{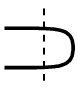} & = \frac{s^\onehalf+s^{-\onehalf}}{t^\onehalf+t^{-\onehalf}} \pic{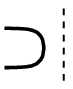}    
    \\
    \pic{trivial-axis} & = 1
  \end{align}
 
  We do not claim that these oriented skein relations suffice 
  to determine the map $W$ uniquely; this is probably false,
  and further relations are necessary to achieve uniqueness.
  %
  %
  In particular the oriented skein relations do not lead to a simple 
  algorithm that calculates $W(D)$ for every diagram $D$.
  This is in contrast to the Jones polynomial,
  for which the oriented skein relation is equivalent
  to the construction via Kauffman's bracket.

  These difficulties suggest that the bracket polynomial
  of Definition \ref{def:TwoVariableBracket} and its defining 
  skein relations \eqref{eq:SkeinA}, \eqref{eq:SkeinB}, and \eqref{eq:SkeinC}
  are the more natural construction in our context.
  For symmetric unions we describe a practical algorithm 
  in Proposition \ref{prop:SymmetricSkein} below.
\end{remark}


\section{The $W$-polynomial of symmetric unions} \label{sec:SymmetricUnionPolynomial}

Having constructed the $W$-polynomial on arbitrary diagrams, 
we now return to symmetric diagrams, and in particular symmetric unions.  
It is in this setting that the $W$-polynomial reveals its true beauty:
integrality (\sref{sub:Integrality}), simple recursion formulae 
(\sref{sub:SymmetricUnionPolynomial}), and special values in $t$ and $s$ 
(\sref{sub:SpecialValues:t}-\sref{sub:SpecialValues:s}).
We continue to use the notation $A^2 = t^{-\onehalf}$ and $B^2 = s^{-\onehalf}$.

\subsection{Integrality} \label{sub:Integrality}

Our first goal is to control the denominator that appears in $W_D$, 
and then to show that this denominator disappears if $D$ is a symmetric union.

\begin{example} \label{exm:WeakIntegrality}
  For the three symmetric diagrams of \fref{fig:SymmetricDiagrams} we find
  \begin{align*}
    W_a(s,t) & = 1 + s^{-1} - s^{-1} (t^{-1}+t^{-3}-t^{-4}) (t+t^3-t^4) , \\
    W_b(s,t) & = s^{\threehalves} \frac{(t^{\onehalf} + t^{-\onehalf})^2}{s^{\onehalf} + s^{-\onehalf}} - s^{2} + s^{3} - s^{4} , \\
    W_c(s,t) & = - s \frac{(t^{\onehalf} + t^{-\onehalf})^2}{s^{\onehalf} + s^{-\onehalf}} - s^{\fivehalves} + s^{\threehalves} .
  \end{align*}
  The symmetry of $D$ implies that $W_D$ is symmetric in $t \leftrightarrow t^{-1}$.
  By specializing $s \mapsto t$ we recover, of course, the Jones polynomials
  of the knot $6_1$, the trefoil knot $3_1$, and the Hopf link $L2a1$, respectively.
  Here we orient the Hopf link (c) such that the reflection along the axis preserves orientations.
\end{example}

We shall see that the symmetry of $D$ also entails that $W_D$ 
has no denominator, apart from $s^{\onehalf} + s^{-\onehalf}$.
The difficulty in proving this integrality of $W_D$ is to find 
a suitable induction argument: resolving a symmetric diagram $D$ 
will lead to asymmetric diagrams, and for asymmetric diagrams 
the desired integrality does not hold in general.

The right setting seems to be the study of ribbon surfaces.
Since this approach introduces its own ideas and techniques
we refer to the article \cite{Eisermann:2009}, whose key result 
is a surprising integrality property of the Jones polynomial:

\begin{theorem}[\cite{Eisermann:2009}] \label{thm:RibbonDivisibility}
  If a link $L \subset \R^3$ bounds a ribbon surface 
  of Euler characteristic $m > 0$,
  then its Jones polynomial $V(L)$ is divisible by 
  $V(\bigcirc^m) = ( - t^{\onehalf} - t^{-\onehalf} )^{m-1}$.
  \qed
\end{theorem}

This is precisely what we need to ensure the integrality of $W_D$:

\begin{corollary}[integrality] \label{cor:WeakIntegrality}
  Let $D$ be a symmetric diagram that perpendicularly 
  traverses the axis in $2m$ points.  Then the bracket polynomial satisfies 
  \begin{equation} \label{eq:WeakIntegralityBracket}
    \bracket{D} \in \Z[A^{\pm1},B^{\pm1}] \cdot (B^{2} + B^{-2})^{m-1} 
  \end{equation}
  and, equivalently, the $W$-polynomial satisfies 
  \begin{equation} \label{eq:WeakIntegrality}
    W_D \in \Z[s^{\pm\onehalf},t^{\pm\onehalf}] 
    \cdot (s^{\onehalf} + s^{-\onehalf})^{m-1} .
  \end{equation}
\end{corollary}

\begin{proof}
  We first consider the case where $D$ has no crossings on the axis.
  By Proposition \ref{prop:NoCrossingsOnAxis} we then know that
  \[ W_D(s,t) = \left(\frac{s^{\onehalf}+s^{-\onehalf}}{t^\onehalf+t^{-\onehalf}}\right)^{m-1} V_L(t) \]
  where $V_L \in \Z[t^{\pm\onehalf}]$ is the Jones polynomial of the link $L$ represented by $D$.
  Using the notation of \sref{sub:SymmetricDiagrams},
  the diagram $D$ has $m$ components of type (a), no components of type (b), 
  and all components of type (c) come in pairs separated by the axis.
  According to Proposition \ref{prop:SymmetricRibbonSurface},
  the link $L$ bounds a ribbon surface of Euler characteristic $m$.
  Theorem \ref{thm:RibbonDivisibility} thus ensures that $V(L)$
  is divisible by $(t^{\onehalf} + t^{-\onehalf} )^{m-1}$,
  so \eqref{eq:WeakIntegrality} holds.

  Both assertions \eqref{eq:WeakIntegralityBracket} and \eqref{eq:WeakIntegrality}
  are equivalent because $\bracket{D}$ and $W_D$ differ only by a writhe normalization 
  of the form $W_D = \bracket{D} \cdot (-A^{-3})^{\alpha(D)} \cdot (-B^{-3})^{\beta(D)}$.
  We can now proceed by induction on the number of crossings on the axis
  using skein relation \eqref{eq:SkeinB}:
  \renewcommand{\pic}[1]{\bracket{\raisebox{-0.9ex}{\includegraphics[height=3ex]{#1}}}}
  \[
  \pic{cross-o-axis} = B^{+1} \pic{cross-h-axis} + B^{-1} \pic{cross-v-axis} , \qquad
  \pic{cross-u-axis} = B^{-1} \pic{cross-h-axis} + B^{+1}\pic{cross-v-axis} .
  \]
  The right hand sides involve only symmetric diagrams, 
  so we can apply our induction hypothesis \eqref{eq:WeakIntegralityBracket}.
  The skein relation thus expresses $\bracket{D}$ as a linear combination 
  of two polynomials in $\Z[A^{\pm1},B^{\pm1}] \cdot (B^{-2} + B^{2})^{m-1}$, 
  so \eqref{eq:WeakIntegralityBracket} holds.
\end{proof}

Notice that for $m=0$ the denominator $s^{\onehalf} + s^{-\onehalf}$ 
is in general unavoidable, as illustrated by Example \ref{exm:WeakIntegrality}.
If the diagram $D$ is symmetric and perpendicularly traverses the axis at least once ($m \ge 1$),
then $W_D$ always is a Laurent polynomial in $s^{\onehalf}$ and $t^{\onehalf}$,
that is, $W_D \in \Z[s^{\pm\onehalf},t^{\pm\onehalf}]$.
This integrality property will be re-proven and strengthened for 
symmetric unions in Corollary \ref{cor:StrongIntegrality} below.

\subsection{Symmetric unions} \label{sub:SymmetricUnionPolynomial}

We will now specialize to symmetric union diagrams,
that is, we assume that each component is of type (a)
as explained in \sref{sub:SymmetricUnions}.

\begin{proposition} \label{prop:SymmetricCrossings}
  Let $D$ be a symmetric union link diagram with $n$ components.
  \begin{enumerate}
  \item
    Each crossing on the axis involves two strands of the same component. \\
    For every orientation, $\aocr$ is a positive crossing and $\aucr$ is a negative crossing.
  \item
    The resolution $\aocr \mapsto \avcr$ yields a symmetric 
    union diagram with $n$ components, while $\aocr \mapsto \ahcr$ 
    yields a symmetric union diagram with $n+1$ components.
  \item
    Each crossing off the axis and its mirror image involve the same components. \\
    Their signs are opposite so that $\alpha(D) = 0$.
  \end{enumerate}
\end{proposition}

\begin{proof}
  The assertions follow from our hypothesis that for a symmetric union
  the reflection $\rho$ maps each component to itself 
  reversing the orientation (see \sref{sub:SymmetricUnions}).
  In particular, each crossing on the axis 
  involves two strands of the same component and
  both strands point to the same halfspace.
  This means that $\aocr$ is necessarily a positive 
  crossing ($\apcr$ or \reflectbox{$\ancr$}),
  while $\aucr$ is necessarily a negative 
  crossing ($\ancr$ or \reflectbox{$\apcr$}).
  The rest is clear.
\end{proof}

In particular, the pairwise linking numbers of 
the components of a symmetric union $D$ vanish.
This also follows from the more geometric construction of 
ribbon surfaces in Proposition \ref{prop:SymmetricRibbonDisks}.
In general, even for symmetric diagrams, the linking number 
need not vanish (see Remark \ref{rem:SymmetricLinking}).

\begin{corollary}
  For every symmetric union link diagram $D$
  the polynomial $W(D)$ is invariant under 
  orientation reversal of any of the components.  
  In other words, $W(D)$ is well-defined 
  for unoriented symmetric union diagrams.
  \qed
\end{corollary}

Even when working with \emph{unoriented} diagrams, 
by Proposition \ref{prop:SymmetricCrossings} we already know 
the $B$-writhe and can anticipate the $B$-normalization.
This observation can be reformulated in the following 
normalized skein relations, which allow for a recursive calculation
of $W(D)$ for every symmetric union diagram $D$:

\begin{proposition} \label{prop:SymmetricSkein}
  Consider a symmetric union diagram $D$ representing a link $L$ 
  with $n$ components.  If $D$ has no crossings on the axis then
  \begin{equation} \label{eq:SUskein0}
    W_D(s,t) = \left(\frac{s^{\onehalf}+s^{-\onehalf}}{t^{\onehalf}+t^{-\onehalf}}\right)^{n-1} V_L(t) ,
  \end{equation}
  where $V_L(t)$ is the Jones-polynomial of the link $L$. 

  If $D$ has crossings on the axis, 
  then we can apply the following recursion formulae:
  \renewcommand{\pic}[1]{W{\left(\raisebox{-0.9ex}{\includegraphics[height=3ex]{#1}}\right)}}
  \begin{align}
    \label{eq:SUskein1}
    \pic{cross-o-axis} &= -s^{+\onehalf} \pic{cross-h-axis} - s^{+1} \pic{cross-v-axis} ,
    \\
    \label{eq:SUskein2}
    \pic{cross-u-axis} &= -s^{-\onehalf} \pic{cross-h-axis} - s^{-1} \pic{cross-v-axis} .
  \end{align}
\end{proposition}

\begin{proof}
  Equation \eqref{eq:SUskein0} follows from Proposition \ref{prop:NoCrossingsOnAxis}:
  since $D$ is a symmetric union, we know that $m = n$.
  If $D$ has crossings on the axis, then we apply 
  the skein relation \eqref{eq:SkeinB} suitably normalized
  according to Proposition \ref{prop:SymmetricCrossings}.
  This proves Equations \eqref{eq:SUskein1} and \eqref{eq:SUskein2}.
\end{proof}

For symmetric unions we can strengthen Corollary \ref{cor:WeakIntegrality}
in the following form:

\begin{corollary}[strong integrality] \label{cor:StrongIntegrality}
  If $D$ is a symmetric union knot diagram, then 
  $W_D$ is a Laurent polynomial in $s$ and $t$.
  More generally, if $D$ is a symmetric union diagram with $n$ components, then
  $W_D \in \Z[s^{\pm1},t^{\pm1}] \cdot (s^{\onehalf}+s^{-\onehalf})^{n-1}$.
\end{corollary}

\begin{proof}
  Every symmetric union diagram $D$ represents a ribbon link $L$.
  If $D$ has no crossings on the axis, then the assertion 
  follows from Equation \eqref{eq:SUskein0} and the divisibility 
  is ensured by Theorem \ref{thm:RibbonDivisibility}.
  We can then proceed by induction on the number of crossings
  on the axis, using Equations \eqref{eq:SUskein1} and \eqref{eq:SUskein2}.
  Notice that $\aocr$, $\aucr$, $\avcr$ have the same number 
  of components, whereas $\ahcr$ has one more component.
\end{proof}

\subsection{Special values in $t$} \label{sub:SpecialValues:t}

A few evaluations of the Jones polynomial have been identified
with geometric data, and some of these can be recovered for the $W$-polynomial:

\begin{proposition}
  Let $D$ by a symmetric union diagram with $n$ components.  
  We have $W_D(s,\xi) = (-s^{\onehalf}-s^{-\onehalf})^{n-1}$
  for every $\xi \in \{ 1, \pm i, e^{\pm 2 i \pi / 3 } \}$,
  and $\frac{\partial W_D}{\partial t}(s,1) = 0$.
\end{proposition}

\begin{proof}
  We proceed by induction on the number of crossings on the axis.
  If $D$ has no crossings on the axis, then we can use
  Equation \eqref{eq:SUskein0} and calculate $W_D(s,t)$ 
  from the Jones polynomial $V_L(t)$.  For the latter we know that
  \begin{align*}
    V_L(1) &= (-2)^{n-1}
    \\
    V_L( e^{\pm 2 i \pi / 3 } ) & = 1
    \\ 
    V_L( \pm i ) &= (-\sqrt{2})^{n-1} (-1)^{\arf(L)}
    \\
    \frac{d V_L}{dt} (1) &= 3 \lk(D) (-2)^{n-1}
  \end{align*}
  (See \cite{LickorishMillet:1986}
  or \cite[Table 16.3]{Lickorish:1997}.)
  Here $\arf(L)$ is the Arf invariant of $L$, 
  and $\lk(D)$ is the total linking number of $L$, 
  i.e., the sum $\sum_{j<k} \lk(L_j,L_k)$ of all pairwise 
  linking numbers between the components $L_1,\dots,L_n$ of $L$.
  Both $\arf(L)$ and $\lk(L)$ vanish because $L$ is a ribbon link.
  The above values of $V_L(\xi)$ thus show that 
  $W_D(s,\xi) = (-s^{\onehalf}-s^{-\onehalf})^{n-1}$ and 
  $\frac{\partial W_D}{\partial t}(s,1) = 0$. 

  If $D$ has at least one crossing on the axis, then we can resolve 
  it according to the skein relation \eqref{eq:SUskein1} or \eqref{eq:SUskein2}.
  More explicitly, consider a positive crossing on the axis:
  \renewcommand{\pic}[1]{W{\left(\raisebox{-0.9ex}{\includegraphics[height=3ex]{#1}}\right)}}
  \[
  \pic{cross-o-axis} = -s^{\onehalf} \, \pic{cross-h-axis} - s \, \pic{cross-v-axis}
  \]
  Notice that $\ahcr$ and $\avcr$ are symmetric union diagrams 
  with $n+1$ and $n$ components, respectively.  We can thus apply 
  the induction hypothesis: for $t = \xi$ we find
  \[
  \pic{cross-o-axis} 
  = -s^{\onehalf} \bigl(-s^\onehalf-s^{-\onehalf}\bigr)^{n} 
  - s \bigl(-s^\onehalf-s^{-\onehalf}\bigr)^{n-1}
  = \bigl(-s^\onehalf-s^{-\onehalf}\bigr)^{n-1} 
  \]
  Likewise, 
  \[
  \frac{\partial}{\partial t} \pic{cross-o-axis} 
  = -s^{\onehalf} \frac{\partial}{\partial t} \pic{cross-h-axis} 
  - s \frac{\partial}{\partial t} \pic{cross-v-axis}
  \]
  and for $t = 1$ all three derivatives vanish.
  Analogous arguments hold when we resolve a negative crossing $\aucr$ 
  instead of a positive crossing $\aocr$.  This concludes the induction.
\end{proof}

\subsection{Special values in $s$} \label{sub:SpecialValues:s}

The following specializations in $s$ are noteworthy:

\begin{proposition}
  For every diagram $D$ the specialization $s \mapsto t$ yields the Jones polynomial 
  of the link $L$ represented by the diagram $D$, that is, $W_D(t,t) = V_L(t)$.
\end{proposition}

\begin{proof}
  For $s \mapsto t$ we no longer distinguish the crossings on the axis,
  and the above skein relations become the well-known axioms
  for the Jones polynomial, thus $W_D(t,t) = V_K(t)$.

  Another way to see this is to start from our two-variable bracket polynomial.  
  For $B \mapsto A$ this becomes Kauffman's bracket polynomial in one variable $A$.  
  Suitably normalized and reparametrized with $t = A^{-4}$ 
  it yields the Jones polynomial, as desired.
\end{proof}

\begin{proposition} \label{prop:ProductFormula}
  If $D$ is the symmetric union knot diagram with partial knots $K_-$ and $K_+$,
  then the specialization $s \mapsto -1$ yields $W_D(-1,t) = V_{K_-}(t) \cdot V_{K_+}(t)$.
  If $D$ is a symmetric union link diagram with $n \ge 2$ components, then $W_D(-1,t) = 0$.
\end{proposition}

\begin{proof}
  The specialization $s \mapsto -1$ means that $s^\onehalf + s^{-\onehalf} = 0$.
  We can now proceed by induction on the number $c$ of crossings on the axis.
  If $c=0$ then the assertion follows from Equation \ref{eq:SUskein0}.
  If $c \ge 1$ then the skein relations \ref{eq:SUskein1} and \ref{eq:SUskein2}, 
  specialized at $s=-1$, show that $W(\aocr) = W(\aucr) = W(\avcr)$.
  This operation reduces $c$ but preserves the number $n$ of components.
  For $n=1$ it also preserves the partial knots $K_\pm$.
\end{proof}

\begin{corollary}
  Suppose that $D$ is the symmetric union knot diagram of two partial knots $K_-$ and $K_+$.
  For $s = t = -1$ we obtain $W_D(-1,-1) = \det(K) = \det(K_-) \cdot \det(K_+)$.
\end{corollary}

\begin{proof}
  The evaluations are subsumed in the following commutative diagram:
  \begin{equation} \label{eq:EvaluationSquare}
    \begin{CD}
      W_D(s,t) \in @. \Z[s^{\pm1},t^{\pm1}] @>{s \mapsto t}>> \Z[t^{\pm1}] @. \ni V_K(t) \\
      @. @V{s \mapsto -1}VV @VV{t \mapsto -1}V @. \\
      V_{K_-}(t) \cdot V_{K_+}(t) \in @. \Z[t^{\pm1}] @>>{t \mapsto -1}> \Z @. \ni \det(K)
    \end{CD}
  \end{equation}

  On the one hand, substituting first $s=-1$ and then $t=-1$ yields $\det(K_-) \cdot \det(K_+)$.
  On the other hand, substituting first $s=t$ and then $t=-1$ yields $\det(K)$.
  The equality $\det(K) = \det(K_-) \det(K_+)$ now follows from $W_D \in \Z[s^{\pm1},t^{\pm1}]$, 
  the integrality property of Corollary \ref{cor:StrongIntegrality}, 
  which ensures the commutativity of Diagram \eqref{eq:EvaluationSquare}.
\end{proof}

The product formula $\det(K) = \det(K_-) \cdot \det(K_+)$ was first proven 
by Kinoshita and Terasaka \cite{KinoshitaTerasaka:1957} in the special case 
that they considered; the general case has been established by Lamm \cite{Lamm:2000}.
We derive it here as a consequence of the more general product formula
for the Jones polynomial established in Proposition \ref{prop:ProductFormula}.

\begin{example}
  The symmetric union diagram of \fref{fig:SymVsAsym}a
  represents the knot $8_{20}$ with partial knots $3_1$ and $3_1^*$.  
  Here we find
  \begin{align*}
    W(s,t) & = 1 - s^{-2} + s^{-2}(t+t^3-t^4)(t^{-1}+t^{-3}-t^{-4}) ,
    \\
    W(t,t) & = V(8_{20}) = -t^{-5}+t^{-4}-t^{-3}+2t^{-2}-t^{-1}+2-t ,
    \\
    W(-1,t) & = V(3_1) \cdot V(3_1^*) = (t+t^3-t^4)(t^{-1}+t^{-3}-t^{-4}) .
  \end{align*}
  In particular $W$ has no denominator and
  is thus an honest Laurent polynomial in $s$ and $t$.
  As it must be, for $t=-1$ the last two polynomials 
  both evaluate to $W(-1,-1) = 9$.
\end{example}

\begin{figure}[hbtp]
  \centering
  \hfill
  \subfigure[$8_{20}$ as symmetric union]{\includegraphics[width=24ex]{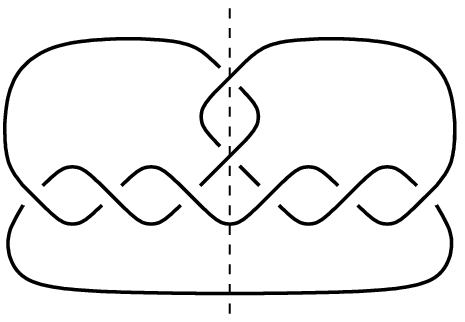}}
  \hfill
  \subfigure[$8_5$ as asymmetric union]{\includegraphics[width=24ex]{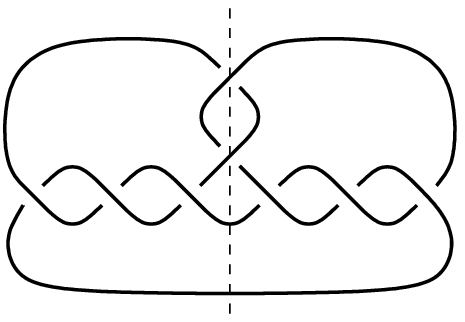}}
  \hfill{}
  \caption{Symmetric vs asymmetric union diagrams} \label{fig:SymVsAsym}
\end{figure} 

\begin{example}
  We should point out that the integrality of $W_D(s,t)$ is a crucial ingredient:
  The asymmetric union depicted in \fref{fig:SymVsAsym}b represents the knot $8_5$ 
  with partial knots $3_1$ and $3_1$.  The lack of symmetry 
  is reflected by a non-trivial denominator in the $W$-polynomial:
  \[
  W(s,t) = \frac{t^{9}-t^{8}+t^{7}-t^{6}+t^{5}+t^{3}
  + s^{-2}(-2t^{7}-t^{5}+2t^{4}+t^{2}) }{t+1} .
  \]
  From this we can recover the Jones polynomial
  \[
  W(t,t) = V(8_5) = 1-t+3t^{2}-3t^{3}+3t^{4}-4t^{5}+3t^{6}-2t^{7}+t^{8}
  \]
  and the determinant $\det(8_5) = 21$.
  If we first set $s = -1$, however, we find 
  the product $W(-1,t) = V(3_1) \cdot V(3_1)$,
  and for $t =  -1$ this evaluates to $\det(3_1) \cdot \det(3_1) = 9$.
\end{example}

This example shows that the evaluation of $W(-1,-1)$ 
is in general not independent of the order of specializations. 
In other words, Diagram \eqref{eq:EvaluationSquare} does \emph{not} 
necessarily commute when we consider rational fractions $W_D \in \Z(s,t)$.
For every diagram $D$ both specializations $W_D(t,t)$ and $W_D(-1,t)$
are Laurent polynomials in $\Z[t^{\pm1}]$.  In $(-1,-1)$ the rational 
function $\R^2 \to \R$ defined by $(s,t) \mapsto W_D(s,t)$ thus has limits 
\[
\lim_{t \to -1} W_D(t,t) = \det(K) 
\quad\text{and}\quad
\lim_{t \to -1} W_D(-1,t) = \det(K_+) \det(K_-) .
\]

If $W_D$ is continuous in $(-1,-1)$ then these two limits co\"incide;
otherwise they may differ, in which case $\det(K) \ne \det(K_-) \cdot \det(K_+)$
as in the preceding example.


%


\section{Examples and applications} \label{sec:Examples}

In this final section we present the computation of some $W$-polynomials.
We begin with preliminaries on alternating knots (\sref{sub:AlternatingKnots})
and a computational lemma (\sref{sub:Computation}).  We then calculate 
the $W$-polynomials of symmetric union diagrams for all ribbon knots 
up to $10$ crossings (\sref{sub:SmallRibbonKnots}) 
and analyze two infinite families of symmetric union diagrams 
for two-bridge ribbon knots (\sref{sub:TwoBridgeRibbonKnots}).

\begin{notation}
  Certain polynomials occur repeatedly in the following calculations.
  In order to save space we will use the abbreviation $u = -s^{\onehalf} - s^{-\onehalf}$ 
  and the auxiliary polynomials $e(t), f(t), \dots$ defined in 
  Table \ref{tab:AuxiliaryPolynomials} on page \pageref{tab:AuxiliaryPolynomials}.
\end{notation}

\subsection{Alternating knots} \label{sub:AlternatingKnots}

A non-trivial symmetric union knot diagram is never alternating.
To see this, start from a point where the knot perpendicularly 
traverses the axis and then travel symmetrically in both directions:
the first crossing-pair is mirror-symmetric and thus non-alternating.

If a knot $K$ admits a reduced alternating diagram with $c$ crossings
then $c$ is the minimal crossing number and every minimal diagram 
representing $K$ with $c$ crossings is necessarily reduced and alternating 
\cite{Kauffman:1987,Murasugi:1987,Thistlethwaite:1987,Turaev:1987}.
This implies the following observation:

\begin{proposition}
  Let $K$ be a prime alternating knot with $c$ crossings.
  If $K$ can be represented by a symmetric union diagram,
  then at least $c+1$ crossings are necessary.
  \qed
\end{proposition}

This explains why in most of our examples 
the symmetric union representations require
slightly more crossings then the (more familiar)
minimal crossing representations.
This argument no longer holds for non-alternating knots:
the example $8_{20}$ in \fref{fig:RibbonSingularity} shows that 
a symmetric union diagram can realize the minimal crossing number.

In the context of alternating diagrams, the span of the Jones polynomial 
turned out to be a fundamental tool and has thus been intensively studied.  

\begin{proposition}
  Let $D$ be a symmetric union diagram with 
  $n$ components having $2c$ crossings off the axis.
  Then the $t$-span of $W_D$ is at most $2c+1-n$.
  It is equal to $2c$ if and only if $n=1$ and 
  the partial diagrams $D_\pm$ are alternating 
  so that $\operatorname{span} V(K_\pm) = c$.
\end{proposition}

\begin{proof}
  The assertion follows from Proposition \ref{prop:SymmetricSkein}
  and the known property of the span of the Jones polynomial
  \cite{Kauffman:1987,Murasugi:1987,Thistlethwaite:1987,Turaev:1987}.
\end{proof}

\begin{proposition} \label{prop:Span:s}
  Suppose that $D$ is a symmetric union diagram with $n$ components
  having $c_+$ positive crossings and $c_-$ negative crossings on the axis.
  Then the degree in $s$ ranges (at most) from 
  $-\frac{n-1}{2} - c_-$ to $\frac{n-1}{2} + c_+$.
\end{proposition}

\begin{proof}
  If $c_+ = c_- = 0$ then the assertion 
  follows from Equation \eqref{eq:SUskein0}.
  We conclude by induction using Equations 
  \eqref{eq:SUskein1} and \eqref{eq:SUskein2}.
\end{proof}

\subsection{A computational lemma} \label{sub:Computation}

As an auxiliary result, we study the effect on $W(D)$ of inserting 
$k$ consecutive crossings and $r$ necklaces on the axis:
the resulting diagram $D_{k,r}$ is shown in \fref{fig:necklace_tangles}.
A positive twist number $k$ stands for crossings of type $\aocr$
and a negative $k$ for crossings of type $\aucr$ because 
both orientations either point from left to right or both point from right to left.

\begin{figure}[htbp]
  \centering
  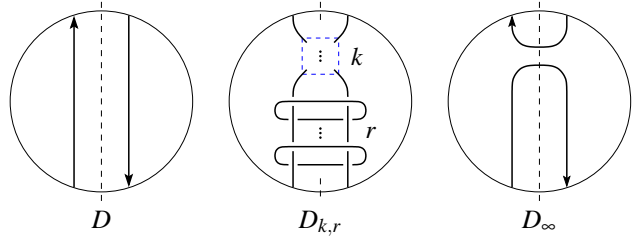
  \caption{The insertion of $k$ crossings and of $r$ necklaces}
  \label{fig:necklace_tangles}
\end{figure}

We assume that $D = D_{0,0}$ is a symmetric union diagram with $n$ components.
By Proposition \ref{prop:Integrality} we can write $W(D) = u^{n-1} \bigl( 1 + d(s,t) \bigr)$
for some polynomial $d(s,t) \in \Z[s^{\pm1},t^{\pm1}]$.

\begin{lemma} \label{lem:Twists}
  If $D_\infty$ is the trivial $(n+1)$-component link then 
  \[
  W_{k,r}(s,t) = u^{n+r-1} \bigl[ 1 
  + (-s)^{k} \cdot (t-1+t^{-1})^r \cdot d(s,t) \bigr] .
  \]
\end{lemma}

\begin{proof}
  \textit{Insertion of necklaces:}
  For arbitrary link diagrams $D = D_{0,0}$, $D_{0,1}$ and $D_\infty$ 
  related as in \fref{fig:necklace_tangles} by insertion of one necklace,
  the $W$-polynomials satisfy the relationship
  \[
  W_{0,1} = (-s^{\onehalf}-s^{-\onehalf}) \cdot (t-1+t^{-1}) \cdot W_D
  \; - \; (t-2+t^{-1}) \cdot W_\infty .
  \]
  If $D_\infty$ is the trivial $(n+1)$-component link, 
  then for $r=1$ we obtain
  \begin{align*}
    W_{0,1}(s,t) & = u^n \cdot (t-1+t^{-1}) \cdot \bigl( 1 + d(s,t) \bigr) 
    - (t-2+t^{-1}) \cdot u^n  \\
    & = u^{n+r-1} \bigl( 1 + (t-1+t^{-1}) \cdot d(s,t) \bigr).
  \end{align*}
  The general case for $r$ necklaces follows by induction.
  
  \textit{Insertion of crossings:}
  We first assume that $k\ge0$ and use induction. 
  For $k=0$ the assertion is valid for all $r\ge0$ and $n\ge1$. 
  For the induction step we assume that the assertion 
  holds for $k-1$ for all $r\ge0$ and $n\ge1$. 
  Then, by Proposition \ref{prop:SymmetricSkein} we have 
  \begin{align*}
    W_{k,r}(s,t) & = -s^{\onehalf} \, u^{n+r} - s \, W_{k-1,r}(s,t) \\
    & = -s^{\onehalf} \, u^{n+r} - s \, u^{n+r-1}
    \bigl( 1 + (-s)^{k-1} (t-1+t^{-1})^r d(s,t) \bigr) \\  
    & = u^{n+r-1} \bigl[ 1 + (-s)^{k} (t-1+t^{-1})^r d(s,t) \bigr].
  \end{align*}
  For $k\ge0$ this completes the proof by induction.
  For $k\le0$ the calculation is analogous.
\end{proof}

As an illustration we calculate the $W$-polynomials 
of two families of symmetric union diagrams. 
They will also be used for the two-bridge knot 
examples in \sref{sub:TwoBridgeRibbonKnots} below.

\begin{example} \label{exm:knot-3_1-4_1}
  The diagrams $D_r$ and $D'_r$ depicted in \fref{fig:necklaces_diagrams} 
  represent the symmetric unions $3_1\sharp 3_1^*$ and $4_1\sharp 4_1$, 
  respectively, with $r$ necklaces.  Their $W$-polynomials are:
  \begin{align*}
    W_{D_r} (s,t) &= u^r \;[ 1 - (t-1+t^{-1})^r \cdot e(t) ] , \\
    W_{D'_r}(s,t) &= u^r \;[ 1 + (t-1+t^{-1})^r \cdot f(t) ] .
  \end{align*}
  This follows from Lemma \ref{lem:Twists}
  and $W_{D_0} (s,t) = 1-e(t)$ and $W_{D'_0}(s,t) = 1+f(t)$.
\end{example}
    
\begin{figure}[hbtp]
  \centering
  \hfill
  \subfigure[$D_r$]{\includegraphics[width=2.5cm]{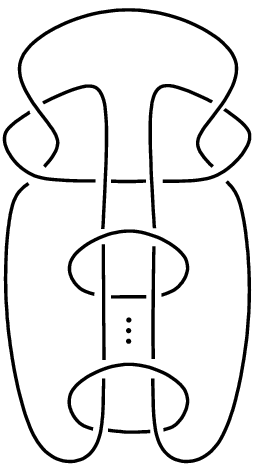}}
  \hfill
  \subfigure[$D'_r$]{\includegraphics[width=2.5cm]{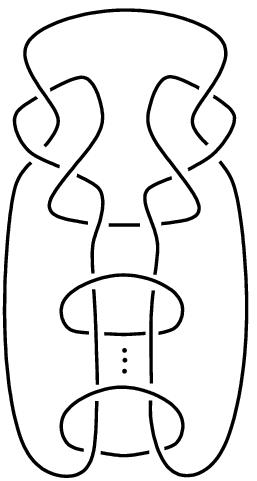}}
  \hfill{}
  \caption{Insertion of $r$ necklaces in diagrams of $3_1\sharp 3_1^*$ and $4_1\sharp 4_1$}
  \label{fig:necklaces_diagrams}
\end{figure}

\subsection{Ribbon knots with at most $10$ crossings} \label{sub:SmallRibbonKnots}

We first study the $6_1$-type family and the Kinoshita-Terasaka family
of symmetric union knot diagrams, and then turn to 
the remaining ribbon knots with at most 10 crossings.

\begin{example} \label{exm:knot-6_1}
  The family of symmetric union diagrams $D_k$ depicted 
  in \fref{fig:SymmetricFamilies}a represents the knots
  $3_1 \sharp 3_1^*$, $6_1$, $8_{20}$, $9_{46}$, $10_{140}$, \dots
  with partial knots $K_+ = 3_1$ and $K_- = 3_1^*$.
  We have $W_0(s,t) = 1 + \bigl( V_{K_+}(t) V_{K_-}(t) - 1 \bigr)$,
  and thus by Lemma \ref{lem:Twists} the $W$-polynomial of $D_k$ is 
  \begin{equation}
    \label{eq:SimplePolynomial}
    W_k(s,t) = 1 + (-s)^{k} \cdot \bigl( V_{K_+}(t) \cdot V_{K_-}(t) - 1 \bigr)
  \end{equation}
  where $V_{K_+}(t)=t^{-1}+t^{-3}-t^{-4}$ and $V_{K_-}(t)=t+t^3-t^4$. 
\end{example}


\begin{figure}[hbtp]
  \centering
  \hfill
  \subfigure[The $6_1$-type family]{\figbox{25ex}{25ex}{\includegraphics[width=18ex]{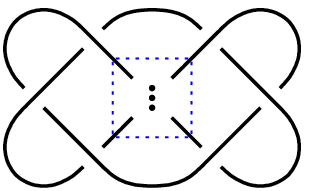}}}
  \hfill
  \subfigure[The Kinoshita-Terasaka family]{\figbox{25ex}{25ex}{\includegraphics[width=18ex]{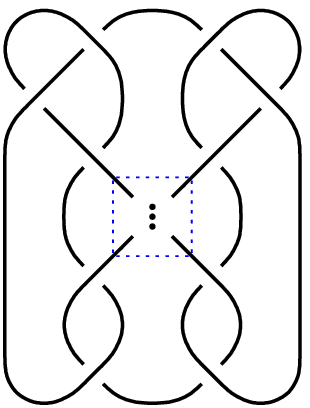}}}
  \hfill{}
  \caption{Two families of symmetric union diagrams}
  \label{fig:SymmetricFamilies}
\end{figure} 

\begin{example} \label{exm:KinoshitaTerasaka}
  The family of symmetric union diagrams $D_k$ depicted 
  in \fref{fig:SymmetricFamilies}b has trivial partial knots;
  $D_0$ represents the trivial knot, $D_1$ represents $10_{153}$,
  and $D_2$ represents the Kinoshita-Terasaka knot. 
  For this family of diagrams Lemma \ref{lem:Twists} 
  is not applicable because $D_\infty$ is non-trivial.
  A small calculation shows that 
  $W_k(s,t) = 1 + \bigl( (-s)^{k} - 1\bigr) \cdot f(t)$.
\end{example}

\begin{figure}[htbp]
  \centering
  \hfill
  \subfigure[$8_9$: an asymmetrically amphichiral diagram]%
  {\figbox{27ex}{25ex}{\includegraphics[width=17ex]{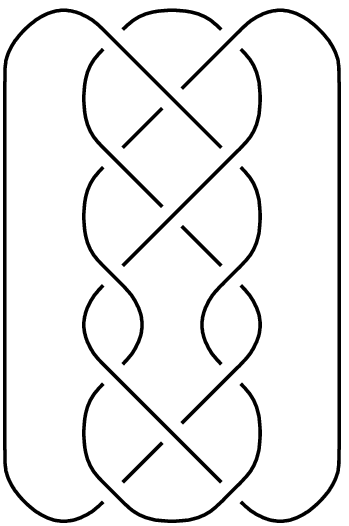}}}
  \hfill
  \subfigure[$8_9$: a symmetrically amphichiral diagram]%
  {\figbox{27ex}{25ex}{\includegraphics[width=17ex]{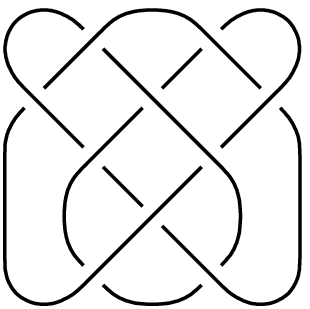}}}
  \hfill{}
  \caption{Two symmetric union diagrams for $8_9$}
  \label{fig:knot-8_9}
\end{figure}

\begin{example} \label{exm:knot-8_9}
  Figure \ref{fig:knot-8_9} displays two symmetric union diagrams for the ribbon knot $8_9$.  
  This knot is amphichiral, and so both diagrams 
  are Reidemeister equivalent to their mirror images.
  But the first diagram (\fref{fig:knot-8_9}a) cannot be 
  symmetrically amphichiral because its $W$-polynomial is not
  symmetric in $s$:
  \[
  W_1(s,t) = 1 + s \, g_2(t) - s^{2} \, f(t).
  \]
  For the second diagram (\fref{fig:knot-8_9}b) we find 
  $
  W_2(s,t) = 1 + f(t),
  $ 
  so that the previous obstruction disappears.
  This diagram is indeed symmetrically amphichiral,
  as shown in \fref{fig:SymmetricAmphichiral}:
  \begin{enumerate}
  \item We start out with a diagram isotopic to \fref{fig:knot-8_9}b,
  \item we slide the upper twist inside-out,
  \item we perform a half-turn of each of the partial knots along its vertical axis, 
  \item we slide the lower twist outside-in,
  \item we turn the entire diagram upside-down.
  \end{enumerate}
  Each of these steps is easily seen to be composed of symmetric Reidemeister moves;
  the last step is realized by a half-turn around the horizontal axis 
  (realizable by symmetric Reidemeister moves) followed by 
  a half-turn around the vertical axis (flype).
\end{example}

\begin{figure}[htbp]
  \centering
  \includegraphics[width=\linewidth]{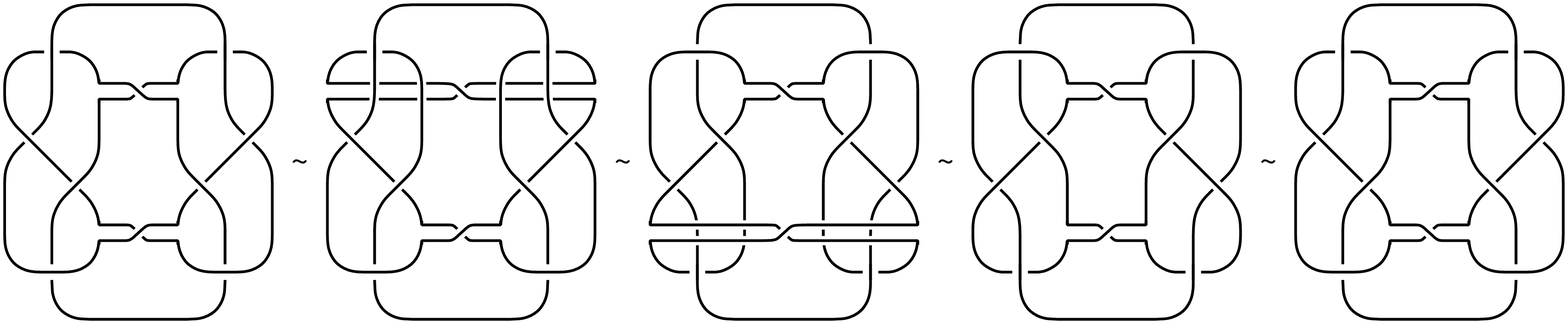}
  \caption{Symmetric equivalence between mirror images}
  \label{fig:SymmetricAmphichiral}
\end{figure}

Table \ref{tab:SmallRibbonKnots} completes 
our list of ribbon knots with at most $10$ crossings. 
In order to save space we have used the auxiliary polynomials
listed in Table \ref{tab:AuxiliaryPolynomials}, which appear repeatedly.

\begin{table}[htbp]
  \renewcommand{\pic}[1]{\raisebox{0pt}[0pt]{\includegraphics[height=8ex]{#1}}}
  \bigskip \noindent
  \begin{tabular}{|c|l|}
    \hline
    diagram & knot \; det \; partial knot \\
    & $W(s,t)$ \\
    \hline
    &\\ & $8_8$ \quad $25$ \quad $4_1$ \\
    & $1 - s   \cdot f(t)$ \\
    \pic{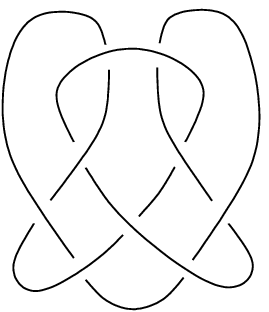} & \\
    &\\ & $10_3$ \quad $25$ \quad $5_1$ \\
    & $1 - s   \cdot g_5(t)$ \\
    \pic{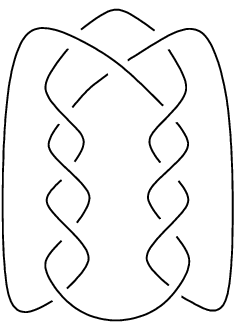} & \\
    &\\ & $10_{22}$ \quad $49$ \quad $5_2$ \\
    & $1 - s   \cdot g_2(t)$ \\
    \pic{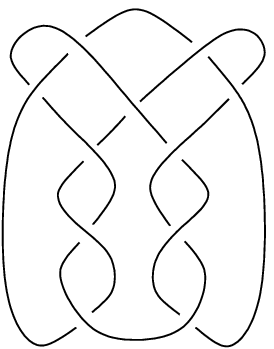} & \\
    &\\ & $10_{35}$  \quad $49$ \quad $5_2$ \\
    & $1 - s   \cdot g_2(t)$ \\
    \pic{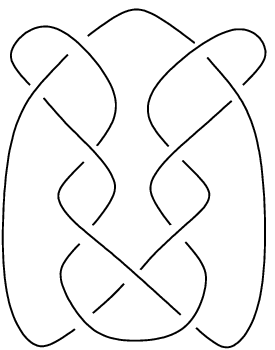} & \\
    \hline
    &\\ & $10_{137}$ \quad $25$ \quad $4_1$ \\
    & $1 + s^2 \cdot f(t)$ \\
    \pic{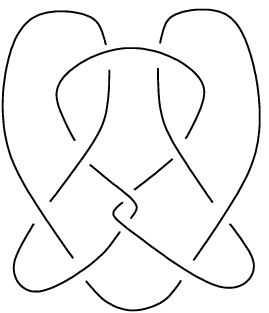} & \\
    &\\ & $10_{129}$ \quad $25$ \quad $4_1$ \\
    & $1 - s   \cdot f(t)$ \\
    \pic{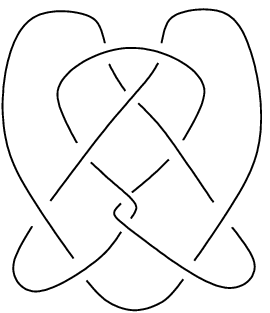} & \\
    &\\ & $10_{155}$ \quad $25$ \quad $4_1$ \\
    & $1 + s^2 \cdot f(t)$ \\
    \pic{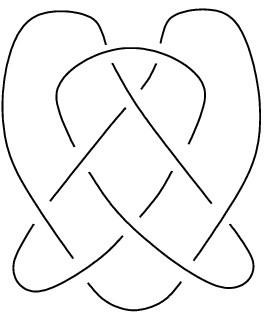} & \\
    \hline
  \end{tabular}
  \quad
  \begin{tabular}{|c|l|}
    \hline
    diagram & knot \; det \; partial knot \\
    & $W(s,t)$ \\
    \hline
    &\\ & $9_{41}$  \quad $49$ \quad $5_2$ \\
    & $1 - s^2 \cdot g_1(t)+s^3 \cdot f(t)$ \\
    \pic{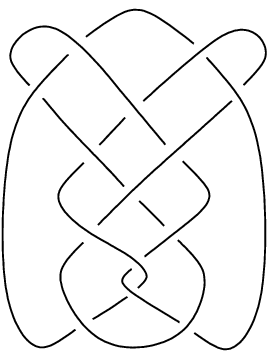} &  \\
    &\\ & $10_{48}$ \quad $49$ \quad $5_2$ \\
    & $1 - g_2(t)$ \\
    \pic{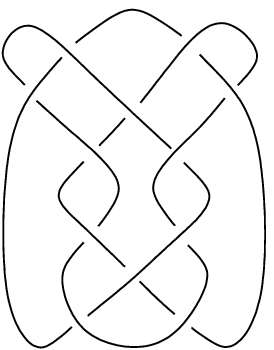} & \\
    \hline
    &\\ & $10_{42}$ \quad $81$ \quad $6_1$ \\
    & $1 - s^{-1} \cdot g_1(t) + h_1(t) - s \cdot g_3(t)$ \\
    \pic{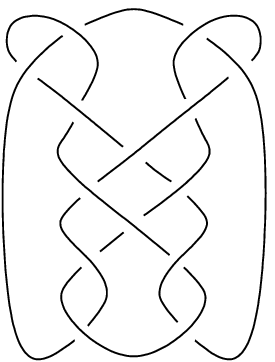} & \\
    &\\ & $10_{75}$ \quad $81$ \quad $6_1$ \\
    & $1 + s^{-2} \cdot g_1(t) - s^{-1} \cdot h_1(t) + g_3(t)$ \\
    \pic{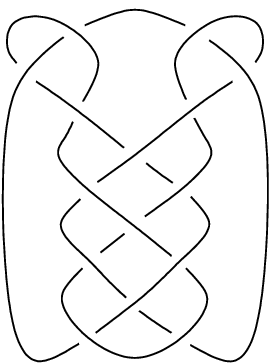} & \\
    &\\ & $10_{87}$ \quad $81$ \quad $6_1$ \\
    & $1 +              g_1(t) - s      \cdot h_1(t) + s^2 \cdot g_3(t)$ \\
    \pic{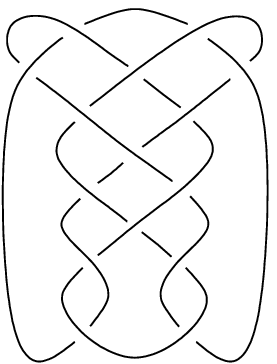} & \\
    &\\ & $10_{99}$ \quad $81$ \quad $6_1$ \\
    & $1 - s^{-1} \cdot   f(t) + h_2(t) - s \cdot   f(t)$ \\
    \pic{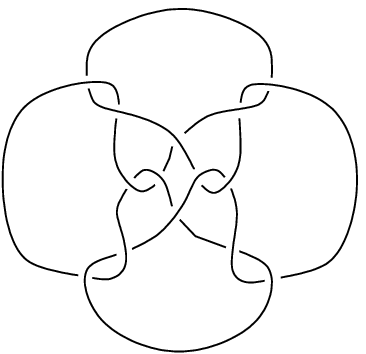} & \\
    \hline
    &\\ & $10_{123}$ \quad $121$ \quad $6_2$ \\
    & $1 + s^{-1} \cdot g_4(t) + h_3(t) + s \cdot g_4(t)$ \\
    \pic{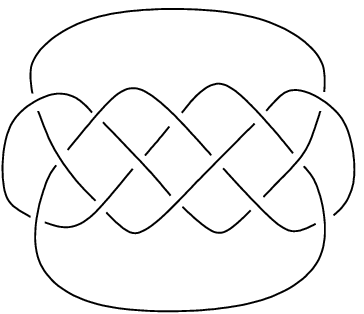} & \\
    \hline
  \end{tabular}
  \bigskip
  \caption{$W$-polynomials of ribbon knots with at most $10$ crossings}
  \label{tab:SmallRibbonKnots}
\end{table}
\begin{table}[htbp]
  \begin{tabular}{ll}
    &\\
    $e(t)$   &= $t^{-3}\,(t^2+1)\;\,(t-1)^2\,(t^2+t+1)$\\
    $f(t)$   &= $t^{-4}\,(t^2+1)\;\,(t-1)^2\,(t^2+t+1)\;\,(t^2-t+1)$\\
    $g_1(t)$ &= $t^{-5}\,(t^2+1)\;\,(t-1)^2\,(t^2+t+1)\;\,(t^2-t+1)^2$\\
    $g_2(t)$ &= $t^{-5}\,(t^2+1)^2  (t-1)^2\,(t^2+t+1)\;\,(t^2-t+1)$\\
    $g_3(t)$ &= $t^{-5}\,(t^2+1)^2  (t-1)^4\,(t^2+t+1)$\\
    $g_4(t)$ &= $t^{-5}\,(t^2+1)\;\,(t-1)^4\,(t^2+t+1)\;\,(t^2-t+1)$\\
    $g_5(t)$ &= $t^{-5}\,(t^2+1)\;\,(t-1)^2\,(t^2+t+1)^2(t^2-t+1)$\\
    $h_1(t)$ &= $t^{-6}\,(t^2+1)\;\,(t-1)^2\,(t^2+t+1)\;\,(t^2-t+1)^3$\\
    $h_2(t)$ &= $t^{-6}\,(t^2+1)\;\,(t-1)^4\,(t^2+t+1)\;\,\phantom{(t^2-t+1)}\,(t^4-t^3+3t^2-t+1)$\\
    $h_3(t)$ &= $t^{-6}\,(t^2+1)\;\,(t-1)^2\,(t^2+t+1)\;\,(t^2-t+1)\,(t^4-3t^3+5t^2-3t+1)$\\
    &\\
  \end{tabular}
  \caption{Auxiliary polynomials used in the description of $W$-polynomials}
  \label{tab:AuxiliaryPolynomials}
\end{table}

Diagrams for $6_1$, $8_{20}$, $9_{46}$, $10_{140}$ are discussed 
in Example \ref{exm:knot-6_1} within the $6_1$-type family,
further diagrams are discussed for $8_9$ in Example \ref{exm:knot-8_9}, 
for  $9_{27}$ in Example \ref{exm:knot-9_27},
and for $10_{153}$ in Example \ref{exm:KinoshitaTerasaka}.
We remark that the $W$-polynomial of $10_{129}$ 
is the same as that of $8_8$, and the $W$-polynomial 
of $10_{155}$ is the same as that of $10_{137}$, 
in accordance with results of Kanenobu \cite{Kanenobu:1986}
who studied an infinite family containing these knots. 
Lemma \ref{lem:Twists} was used for the diagrams 
of $6_1$, $8_{20}$, $9_{46}$, $10_{140}$ in Example \ref{exm:knot-6_1} 
and again for $8_8$, $10_3$, $10_{22}$, $10_{35}$, $10_{137}$ 
in Table \ref{tab:SmallRibbonKnots}.

\subsection{Two-bridge ribbon knots} \label{sub:TwoBridgeRibbonKnots}

In this final paragraph we establish symmetric inequivalence in the family 
of two-bridge ribbon knots that we studied in \cite{EisermannLamm:2007}.
We consider the symmetric union diagrams $D_n$ and $D'_n$ shown in \fref{fig:familiy-statement}.
They are defined for $n\ge2$ and we write $n=2k+1$ in the odd case and $n=2k$ in the even case.

\begin{figure}[htbp]
  \centering
  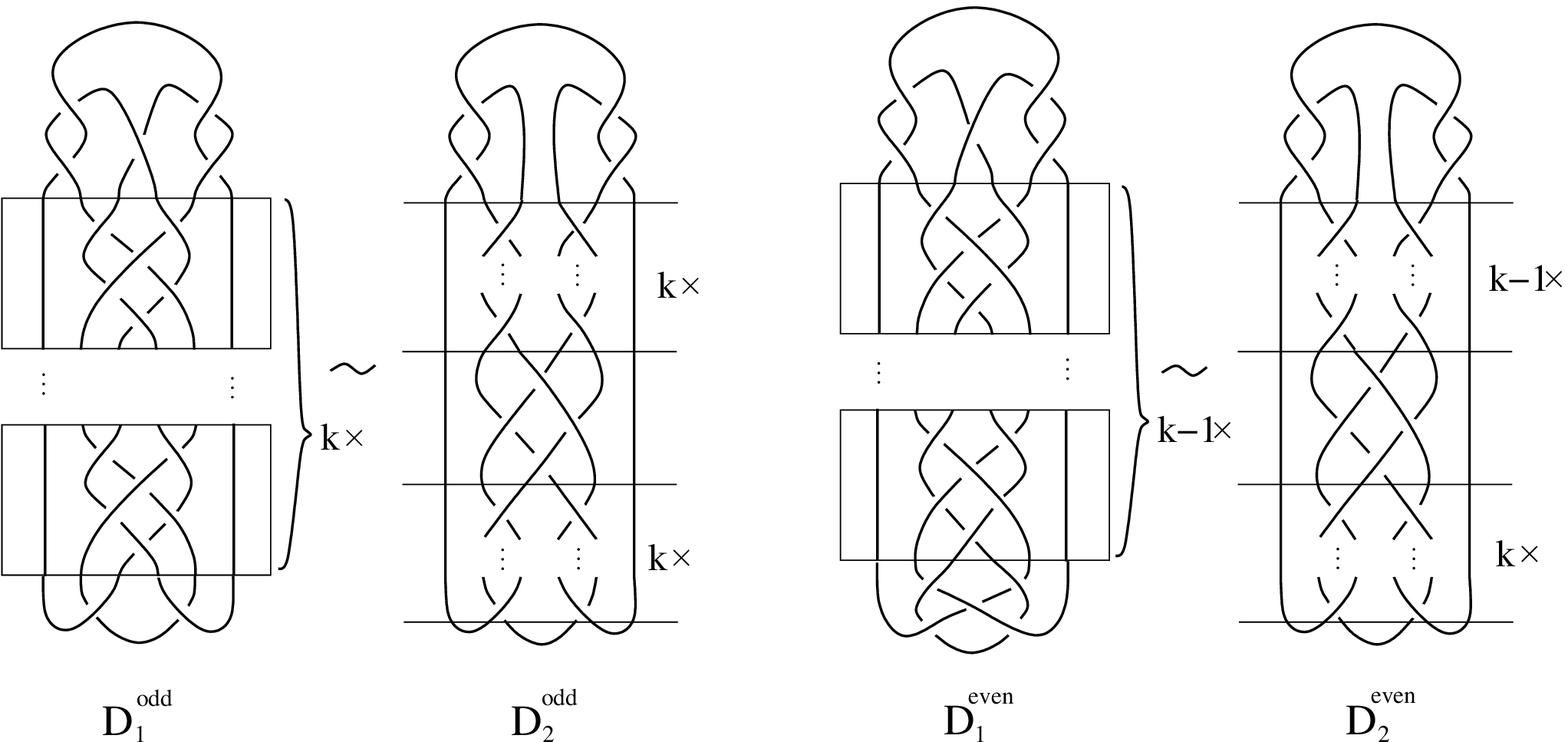
  \caption{The family of knot diagrams $D_n$ and $D'_n$ 
    of Theorem \ref{non-equivalence}}
  \label{fig:familiy-statement}
\end{figure}

\begin{remark}
  The symmetric union diagrams $D_n$ and $D'_n$ represent two-bridge knots 
  of the form $K(a,b)=C(2a,2,2b,-2,-2a,2b)$ with $b=\pm1$.  
  These knots have genus $3$ and their crossing number is $6+n$.
  The first members can be identified as follows:
  $8_9    = K(-1,-1)$ for $n=2$, 
  $9_{27} = K(-1,1)$  for $n=3$, 
  $10_{42}= K(1,1)$   for $n=4$, 
  $11a96  = K(1,-1)$  for $n=5$,
  $12a715 = K(-2,-1)$ for $n=6$, 
  $13a2836= K(-2,1)$  for $n=7$. 

  The diagrams $D_2$ and $D'_2$ are the two mirror-symmetric 
  diagrams of $8_9$ shown in \fref{fig:knot-8_9}b.
  They have been shown to be symmetrically equivalent
  in \fref{fig:SymmetricAmphichiral}.
  
  The diagrams $D_3$ and $D'_3$ are the two symmetric union 
  representations of $9_{27}$ depicted in \fref{fig:Knot-9_27}.
  They have already been proven to be distinct in Example \ref{exm:knot-9_27}.

  We do not know if the diagrams $D_4$ and $D'_4$,
  representing $10_{42}$, are symmetrically equivalent:
  their $W$-polynomials co\"incide but no 
  symmetric transformation has yet been found.
\end{remark}

We have proved in \cite{EisermannLamm:2007}, Theorem 3.2, 
that for each $n$ the symmetric union diagrams $D_n$ 
and $D'_n$ are asymmetrically equivalent.
One of the motivations for developing the $W$-polynomial was to show 
that $D_n$ and $D'_n$ are, in general, not symmetrically equivalent:

\begin{theorem} \label{non-equivalence}
  The symmetric union diagrams $D_n$ and $D'_n$ 
  depicted in \fref{fig:familiy-statement}
  are not symmetrically equivalent if $n=3$ or $n\ge5$.
\end{theorem}

\begin{proof}
  We show that the $W$-polynomials of the two diagrams 
  $D_n$ and $D'_n$ are different for $n=3$ and $n\ge5$.
  By Proposition \ref{prop:Span:s} the degree in $s$ of 
  the $W$-polynomial of $D'_n$ ranges at most from $-1$ to $1$.
  It is enough to show that the maximal or minimal degree in $s$ 
  of the $W$-polynomial of $D_n$ is bigger than 1, or smaller 
  than $-1$, respectively. For brevity, we only analyze the maximal degree.
  
  \medskip\noindent\textit{Odd case:}
  For $n=2k+1$ we claim that $\max\deg_s W(D_n) = k+1$.
  \smallskip

  The diagram $D_n$ contains $k$ negative and $k+1$ positive crossings on the axis, 
  therefore the maximal degree in $s$ is less or equal to $k+1$.
  We resolve all $k$ negative crossings $\aucr$ on the axis to $\ahcr$.
  Only this resolution contributes by Proposition \ref{prop:SymmetricSkein} 
  to the maximal degree $s^{k+1}$ and we obtain a factor of $(-s^{-\onehalf})^k$.
  The resulting diagram is illustrated in \fref{fig:basis_diagrams}a: 
  it has $k$ necklaces and $k+1$ consecutive positive crossings on the axis, 
  for which the horizontal resolution is a trivial link with $k+2$ components. 
  Let $u^k (a_k(t)+1)$ be the $W$-polynomial of the latter diagram 
  without the crossings on the axis, then by Lemma \ref{lem:Twists} 
  the $W$-polynomial of the diagram with $k+1$ crossings is   
  $(-s^{-\onehalf})^k u^k \bigl((-s)^{k+1}a_k(t)+1\bigr)$, 
  including the factor $(-s^{-\onehalf})^k$ from the resolution step.
  By Example \ref{exm:knot-3_1-4_1} we find that $a_k(t)\ne 0$, 
  proving that in the odd case the maximal $s$-degree of $D_n$ is $k+1$. 
  Note that the maximal $s$-degree of $(-s^{-\onehalf})^k u^k$ is zero.
  For odd $n\ge3$ the maximal $s$-degree is therefore greater than $1$.
  
  \begin{figure}[hbtp]
  \centering
  \hfill
  \subfigure[Odd case]{\includegraphics[width=4cm]{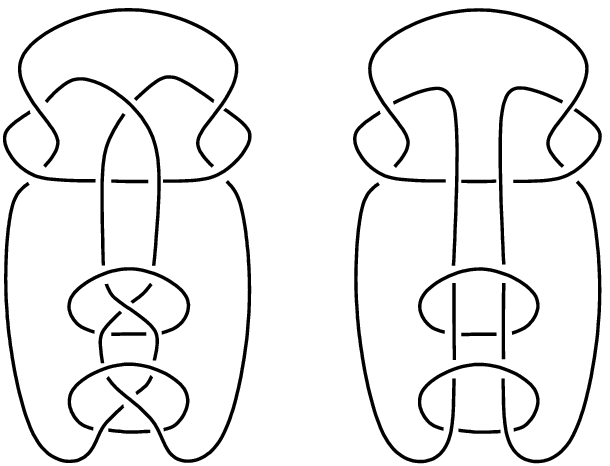}}
  \hfill
  \subfigure[Even case]{\includegraphics[width=4cm]{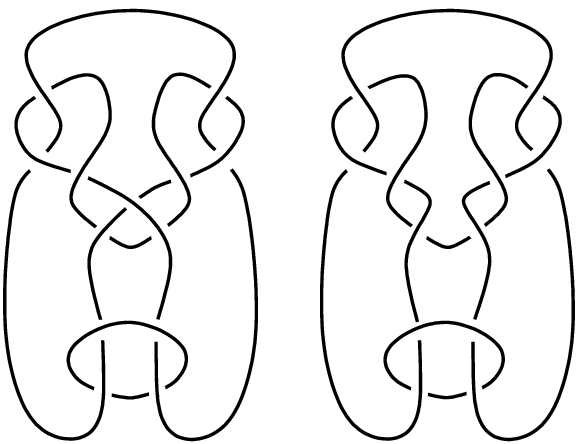}}
  \hfill{}
  \caption{Diagrams occuring in the proof of Theorem \ref{non-equivalence} (for $k=2$)}
  \label{fig:basis_diagrams}
  \end{figure} 

  \medskip\noindent\textit{Even case:}
  For $n=2k$ we claim that $\max\deg_s W(D_n) = k-1$.
  \smallskip

  We observe that the diagram $D^\star_n$ obtained from $D_n$ by deleting 
  the first and the last crossing on the axis has the same $W$-polynomial as $D_n$.  
  This requires a short calculation using the fact that the $\ahcr$-resolutions 
  for these crossings are diagrams of the trivial link. 

  As an illustration, let us make the first three cases explicit.
  For $n=2$ the diagram $D^\star_2$ is $4_1 \sharp 4_1$. 
  For $n=4$ the diagram $D^\star_4$ co\"incides with $D'_4$, showing 
  that $D_4$ and $D'_4$ cannot be distinguished by their $W$-polynomials. 
  For $n=6$ the two diagrams $D_6$ and $D_6^\star$ are illustrated in \fref{fig:knot-12a715}; 
  they represent the knots $12a715$ and $12a3$, respectively. 
  \footnote{We seize the occasion to correct an unfortunate misprint 
    in \cite{EisermannLamm:2007}: the caption of Fig.\ 7 showing 
    a similar diagram states wrong partial knots.  The partial knots 
    of the shown diagrams of $12a3$ are $C(3,4)$ and $C(2,6)$.}
  
  \begin{figure}[htbp]
  \centering
  \includegraphics[scale=0.6]{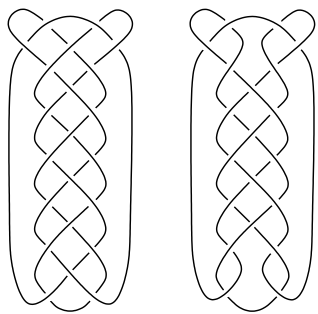}
  \caption{Two symmetric union diagrams sharing the same $W$-polynomial}
  \label{fig:knot-12a715}
  \end{figure}

  By Proposition \ref{prop:Span:s} the exponents 
  of $s$ in $W(D_n)$ lie between $-k$ and $+k$.
  In the diagram $D^\star_n$, however, only $k-1$ negative 
  and $k-1$ positive crossings on the axis remain, 
  so in $W(D_n) = W(D_n^\star)$ the bounds $-k$ and $+k$ are not attained,
  whence $\max\deg_s W(D_n) \le k-1$. 
  
  In $D_n^\star$ we resolve all $k-1$ negative crossings $\aucr$ on the axis to $\ahcr$.
  As in the previous case, only this resolution contributes to 
  the maximal degree $s^{k-1}$ and we obtain a factor of $(-s^{-\onehalf})^{k-1}$.
  The resulting diagram is illustrated in \fref{fig:basis_diagrams}b: 
  it has $k-1$ necklaces and $k-1$ consecutive positive crossings on the axis, for which the 
  horizontal resolution is a trivial link with $k+1$ components.  
  The process of adding necklaces and twists is the same as in the odd case:
  for the $W$-polynomial of the diagram with $k-1$ crossings we have 
  $(-s^{-\onehalf})^{k-1} u^{k-1} \bigl((-s)^{k-1}b_k(t)+1\bigr)$ 
  if the $W$-polynomial of the respective diagram without 
  the twists is $(-s^{-\onehalf})^{k-1} u^{k-1} (b_k(t)+1)$,
  both already including the factor $(-s^{-\onehalf})^{k-1}$.
  Using again Example \ref{exm:knot-3_1-4_1} we find that $b_k(t)\ne 0$. 
  This proves that in the even case the maximal $s$-degree of $W(D_n) = W(D_n^\star)$ is $k-1$. 
  Hence, for even $n\ge6$ the maximal $s$-degree is greater than $1$, which proves the theorem.
\end{proof}


\bibliographystyle{plain}
\bibliography{symjones}

\end{document}